\documentclass{amsart}
\usepackage{amssymb,enumerate}
\usepackage[colorlinks,bookmarks]{hyperref}
\hypersetup{
    colorlinks = true,
    citecolor = blue,
    linkcolor = purple
}
\usepackage{tikz,tikz-cd}
\usetikzlibrary{arrows,positioning,shapes,decorations.pathreplacing}
\usepackage[mathscr]{euscript}
\newcommand{\goaway}[1]{}
\usepackage[margin=1.15in]{geometry}
\usepackage{cleveref}

\usepackage{enumitem}
\newlist{deflist}{enumerate}{1}
\setlist[deflist]{label=(\arabic{deflisti}), ref=\thelemma.(\arabic{deflisti}),noitemsep}

\newtheorem{lemma}{Lemma}[subsection]
\newtheorem{theorem}[lemma]{Theorem}
\newtheorem{proposition}[lemma]{Proposition}
\newtheorem{corollary}[lemma]{Corollary}

\newtheorem{introthm}{Theorem}

\newtheorem{introcor}[introthm]{Corollary}
\theoremstyle{definition}
\newtheorem{example}[lemma]{Example}
\newtheorem{remark}[lemma]{Remark}

\newtheorem{definition}[lemma]{Definition}
\newtheorem{d-a}[lemma]{Definition-Assumption}
\newtheorem{question}[lemma]{Question}

\renewcommand{\P}{\mathscr{P}}      
\newcommand{\Q}{\mathscr{Q}}        
\renewcommand{\L}{\mathscr{L}}		
\newcommand{\zero}{\hat{0}}			
\newcommand{\one}{\hat{1}}			
\newcommand{\A}{\mathscr{A}}		
\newcommand{\B}{\mathscr{B}}        
\newcommand{\qA}{\A/Y}              
\newcommand{\pr}{\bar{p}}           
\newcommand{\G}{\mathbb{G}}         
\newcommand{\Z}{\mathbb{Z}}			
\newcommand{\R}{\mathbb{R}}			
\newcommand{\C}{\mathbb{C}}         
\newcommand{\QQ}{\mathbb{Q}}        
\newcommand{\gp}{\mathfrak{G}}      

\DeclareMathOperator{\Hom}{Hom}
\DeclareMathOperator{\Conf}{Conf}
\DeclareMathOperator{\rk}{rk}
\DeclareMathOperator{\im}{im}
\DeclareMathOperator{\Hor}{Hor}
\DeclareMathOperator{\Bad}{Bad}
\DeclareMathOperator{\Aut}{Aut}

\newcommand{\tr}[1]{\max\P(\A_{#1})}
\newcommand{\st}{\colon}
\newcommand{\angles}[1]{\langle #1 \rangle}
\newcommand{\imp}[2]{\item[\eqref{#1}$\implies$\eqref{#2}]}
\newcommand{\equ}[2]{\item[\eqref{#1}$\iff$\eqref{#2}]}
\newcommand{\solidnodes}{\tikzstyle{every node}=
[draw,circle,fill=black,minimum size=4pt,inner sep=0pt]}
\newcommand{\defh}[1]{\textbf{#1}}
\newcommand{\defhno}[1]{\textbf{#1}}

\makeatletter
\newcommand{\mylabel}[2]{#2\def\@currentlabel{#2}\label{#1}}
\makeatother
\makeatletter
\newcommand\mynobreakpar{%
\par\nobreak\@afterheading} 
\makeatother

\title{
Supersolvable posets and fiber-type abelian arrangements
}
\author{Christin Bibby}
\address{Louisiana State University, Baton Rouge, LA, USA}
\email{\url{bibby@math.lsu.edu}}
\author{Emanuele Delucchi}
\address{University of Applied Arts and Sciences of Southern Switzerland}
\email{\url{emanuele.delucchi@supsi.ch}}



\begin{document}

\begin{abstract}
We present a combinatorial analysis of fiber bundles of generalized configuration spaces on connected abelian Lie groups.
These bundles are akin to those of Fadell--Neuwirth for configuration spaces, 
and their existence is detected by a combinatorial property of an associated finite partially ordered set. This is consistent with Terao's fibration theorem connecting  bundles of hyperplane arrangements to Stanley's lattice supersolvability.
We obtain a combinatorially determined class of $K(\pi,1)$ toric and elliptic arrangements. 
Under a stronger combinatorial condition, 
we prove a factorization of the Poincar\'e polynomial when the Lie group is noncompact. 
In the case of toric arrangements, this provides
an analogue of Falk--Randell's formula relating the Poincar\'e polynomial to the lower central series of the fundamental group.
\end{abstract}

\maketitle

\vspace{-2em}

\section{Introduction}

The fiber bundles studied by Fadell and Neuwirth \cite{FN} provide a fundamental tool in the study of configuration spaces on  manifolds. 
In the case of configurations of points in the plane $\C$,
these bundles can be generalized 
to complements of certain arrangements of hyperplanes in a complex vector space known as ``fiber-type'' \cite{FR}. 
Such arrangement complements
are $K(\pi,1)$s and exhibit a noteworthy relationship between the Poincar\'e polynomial and the lower central series of the fundamental group, first observed in the case of configuration spaces by Kohno \cite{kohno} and proved in general by Falk and Randell \cite{FR}. 
Terao \cite{teraofibthm} showed that a linear hyperplane arrangement is fiber-type if and only if the poset of intersections satisfies a purely combinatorial condition, which was defined by Stanley \cite{stanley} and motivated by the structure of the lattice of subgroups in a supersolvable group.

Much less is known for generalized configuration spaces on manifolds other than Euclidean space.
An {\em abelian  arrangement}  
is a finite set of subgroup cosets in a connected abelian Lie group (\Cref{def:alga}). The complement of the union of all elements of the arrangement inside the ambient Lie group is the manifold of interest here. 
An arrangement is called {\em fiber-type} if either it consists of single points, or if its complement fibers over the complement of a lower-dimensional fiber-type arrangement. The resulting tower of fibrations generalizes the one associated to configuration spaces.

The combinatorial data of an arrangement is the \emph{poset of layers}, i.e., the set 
of all connected components of intersections of elements of the arrangement,
partially ordered by inclusion. The significance of this poset was first recognized by Zaslavsky \cite{TomZ}.
In this context we are able to define a notion of supersolvability (\Cref{def:ss}) based on a suitable extension of the concept of modularity in lattices (\Cref{def:modideal}). A key difference here is that, for some topological consequences, we need a stronger condition which we call \emph{strict supersolvability} (\Cref{def:sss}). 

\begin{introthm}[Fibration Theorem, \Cref{thm:ft=ss} and \Cref{thm:FN}]\label{thm:A}
An essential abelian arrangement is fiber-type if and only if its poset of layers is supersolvable. 
If the poset of layers is strictly supersolvable then the arrangement bundles can be pulled back from Fadell--Neuwirth's bundles of configuration spaces.
\end{introthm}

The latter part of \Cref{thm:A} was first observed by Cohen \cite{cohen} in the case of hyperplane arrangements. 
This allows one to pull back properties of the bundle, such as a description of monodromy or a section of the configuration space bundle (the latter exists for all connected abelian Lie groups, see \Cref{cor:section}).
We leave open these questions when an arrangement is not strictly supersolvable (\Cref{q:monodromy}).

As a motivating example for \Cref{thm:A} we consider Dowling posets as defined by the first author and Gadish in \cite{BG}. 
These arise in the study of orbit configuration spaces -- an equivariant analogue to ordinary configuration spaces -- and generalize the lattices of Dowling \cite{dowling}.
Xicot\'encatl \cite{X} established an equivariant analogue of Fadell--Neuwirth's bundles for orbit configuration spaces, corresponding to the fact that Dowling posets are supersolvable (\Cref{prop:dowling}).

\medskip

Since the work of Brieskorn and Deligne \cite{Brieskorn,Deligne}, a long standing problem in arrangement theory is to classify which arrangement complements are $K(\pi,1)$. As a corollary to \Cref{thm:A}, we obtain a combinatorially determined class of $K(\pi,1)$ toric and elliptic arrangements (i.e., when the Lie group is $\mathbb C^\times$ or $(S^1)^{2}$). 
Using the aforementioned bundle section for strictly supersolvable arrangements, this also yields a description of the fundamental group.

\begin{introcor}[\Cref{cor:kpi1} and \Cref{cor:pi1}]
If the poset of layers of a linear, toric, or elliptic arrangement is supersolvable, then the arrangement complement is a $K(\pi,1)$ space.
If the poset is strictly supersolvable, then the fundamental group is an iterated semidirect product of free groups.
\end{introcor}

On the purely combinatorial side, we give an abstract definition of a {\em geometric poset}  (\Cref{def:geom}) that seems to provide the right level of generality for a study of posets of layers of arrangements, and that reduces to the well-known notion in the case of (semi)lattices, see \cite{WW}.
Geometric posets support an equivalent definition of supersolvability (\Cref{thm:niceM}) which, for posets of layers of affine hyperplane arrangements, agrees with the definition 
given by Falk and Terao \cite{falk-terao}.
We prove that a geometric semilattice is supersolvable if and only if its canonical  extension to a geometric lattice is supersolvable (\Cref{thm:sscone}). 
The topological consequence of this is an affine analogue of Terao's fibration theorem (\Cref{thm:affinefib}).

Stanley \cite{stanley} proved that the characteristic polynomial of a supersolvable lattice has positive integer roots, and we observe the same phenomenon for strictly supersolvable posets.
Through a formula of Orlik and Solomon \cite{OS} for complex hyperplane arrangements and Liu, Tran, and Yoshinaga \cite{LTY} for noncompact abelian Lie group arrangements, this yields a factorization of the Poincar\'e polynomial for arrangement complements when the poset of layers is strictly supersolvable.

\begin{introthm}[Polynomial Factorization, \Cref{thm:charpoly} and \Cref{cor:poincare}]
Let $\P$ be a strictly supersolvable poset. Then there is a partition $A_1\sqcup\cdots\sqcup A_n$ of the atoms of $\P$ such that the characteristic polynomial of $\P$ factors as
\[ \chi_\P(t) = \prod_{i=1}^n (t-|A_i|).\]
If $\P$ is the poset of layers for an essential arrangement in $\G^n$ where $\G=(S^1)^d\times\R^v$ with $v>0$, then 
the Poincar\'e polynomial of the arrangement complement is
\[\operatorname{Poin}(t) = \prod_{i=1}^n \left((1+t)^d+|A_i|t^{d+v-1}\right).\]
\end{introthm}

The poset of layers of every abelian Lie group arrangement $\A$ is the quotient of a geometric semilattice by the action of a free abelian group (see \Cref{ex:UC} for a precise statement). 
We prove that a geometric semilattice is supersolvable if and only if its quotient by a suitable group action is supersolvable (\Cref{thm:quotient}).
Topologically, this relates the fiber-type property of an abelian arrangement complement with that of a covering space (\Cref{cor:coverdown} and \Cref{cor:coverup}). Combinatorially, this shows that our notion of supersolvability is a natural extension of the classical one for matroids to the context of group actions of semimatroids, see \cite{dD}.

\medskip

Now consider a noncompact abelian Lie group and an arrangement bundle for which the algebraic monodromy is trivial. 
This includes all strictly supersolvable toric arrangements (\Cref{rmk:toricmonodromy}). In fact, strict supersolvability is necessary for the monodromy to be trivial, and we obtain a tensor decomposition of the cohomology algebra for such arrangements (\Cref{thm:monodromy}) which can in principle be expressed combinatorially. 
Finally, for toric arrangements, a combination of our 
results about 
Poincar\'e polynomial factorization, existence of a section, and trivial monodromy yields a formula akin to Falk and Randell's \cite{FR} lower central series formula.

\begin{introthm}[Lower Central Series Formula for Toric Arrangements, Thm.\ \ref{thm:LCS}]
Let $\A$ be a strictly supersolvable toric arrangement and let $A_1\sqcup\cdots\sqcup A_n$ be the induced partition of the atoms in its poset of layers. For $j\geq 1$, let $\varphi_j$ be the rank of the $j$th successive quotient in the lower central series of the fundamental group of the complement of $\A$. Then
\[\prod_{j=1}^\infty (1-t^j)^{\varphi_j} = \prod_{i=1}^n (1-(|A_i|+1)t).\]
\end{introthm}

\hyphenation{For-schungs-in-sti-tut}
\subsection{Acknowledgements} The bulk of the work reported in this paper was carried out during a Research-in-Pairs visit by the authors at  Mathematisches Forschungsinstitut Oberwolfach.
Both authors are grateful to MFO and their staff for their hospitality and support, and for the ideal working environment provided. 
The first author was supported by NSF DMS-2204299.
The second author was partly supported by the SNSF professorship grant PP00P2\_150552/1.

\setcounter{tocdepth}{1}
\tableofcontents

\section{Supersolvable locally geometric posets}

We recall basic ideas about posets and supersolvable geometric lattices.
We then introduce the definitions of M-ideal (\Cref{def:modideal}) and supersolvability (\Cref{def:ss}), extending the notions of modular elements in and supersolvability of geometric lattices to our setting of locally geometric posets.
We conclude this section with a motivating example: 
Dowling posets (\Cref{prop:dowling}).

\subsection{Generalities about posets} Let $\P$ be a partially ordered set (or ``poset'') with partial order relation $<$. For $x,y\in \P$ write $x\leq y$ when either $x<y$ or $x=y$, and $x \lessdot y$ when $x<y$ and $x\leq z<y$ implies $x=z$. Given any $x\in \P$ let $\P_{<x}:=\{y\in \P\st y<x\}$, partially ordered by the restriction of $<$. The posets $\P_{\leq x}$, $\P_{>x}$ and $\P_{\geq x}$ are defined analogously.  The \textit{interval} between two elements $x,y\in \P$ is the set $[x,y]:=\P_{\geq x}\cap \P_{\leq y}$. 

Let $\P, \Q$ be posets. A poset morphism $f:\P\to\Q$ is an order-preserving map (i.e., we require that $x\leq y$ implies $f(x)\leq f(y)$ for all $x,y\in \P$). We call $f$ a poset isomorphism if $f$ is bijective and its inverse is a poset morphism. An \textit{automorphism} of a poset $\P$ is any isomorphism $\P\to\P$.

A \textit{chain} in $\P$ is any $C\subseteq \P$ such that either $c_1\leq c_2$ or $c_1\geq c_2$ for all $c_1,c_2\in C$. The \textit{length} of a chain $C$ is $\vert C \vert -1$. The poset $\P$ is \textit{chain-finite} if all chains in $\P$ have finite length. An \textit{antichain} in $\P$ is any subset whose elements are pairwise incomparable. 

Call $\P$  \textit{bounded below} if it contains a unique minimum element, which we denote by $\zero$. In this case,  the \textrm{rank} $\rk(x)$ of an element $x\in \P$ is the maximum length of any chain in $\P_{\leq x}$. The set of \textit{atoms} of a bounded-below poset $\P$ is $$A(\P):=\{x\in \P\st \rk(x)=1\}.$$ If any two maximal chains in the same interval of $\P$ have equal length, $\P$ is said to be \textit{graded}. Equivalently, the assignment $x\mapsto \rk(x)$ defines a function $\rk:\P\mapsto \mathbb Z_{\geq 0}$ such that $\rk(\zero)=0$ and $\rk(y)=\rk(x)+1$ whenever $x\lessdot y$. If such a function exists, it is unique, and it is called the \textit{rank function} of $\P$.  If any two maximal elements in $\P$ have the same rank, then $\P$ is called \textit{pure}.

For any
two elements $x,y\in\P$, we define $x\vee y$ to be the \textit{set} of minimal
upper bounds and $x\wedge y$ to be the \text{set} of maximal lower bounds. That
is:
\[ x\vee y := \min\{z\in\P\st z\geq x \text{ and } z\geq y\},\]
 \[x\wedge y := \max\{z\in\P\st z\leq x \text{ and } z\leq y\}. \]
More generally, denote by $\bigwedge T$ and $\bigvee T$ the sets of minimal upper
bounds and maximal lower bounds of a set $T\subseteq\P$.

A {\em complement} of an element $x$ in a chain-finite poset $\P$ is any $z\in \P$ such that $x\vee z\subseteq \max \P$ and $x\wedge z\subseteq \min\P$. Given a subset $X\subseteq \P$ we say that $z\in\P$ is a complement to $X$ if $z$ is a complement of every $x\in X$. (Notice that this definition generalizes the usual one for lattices.)

\subsection{Locally geometric posets}
Recall that a \textit{lattice} is a poset $\L$ in which any pair of elements $x,y\in\L$ has a unique minimum upper bound ($\vert x\vee y \vert = 1$) and a unique maximum lower bound ($\vert x\wedge y \vert =1$). In this case we abuse notation and write, e.g., $a=x\vee y$ for $a\in x\vee y$. A meet-semilattice is a poset in which any pair of elements has a unique maximum lower bound. Note that any chain-finite meet-semilattice (hence also any chain-finite lattice) is bounded below.

\begin{definition} \label{def:geomlattice}
A chain-finite lattice $\L$ is called \defhno{geometric} if and only if, for all $x,y\in \L$:
\begin{center}
	$x\lessdot y$ if and only if there is an atom $a\in A(\L)$ with $a\not\leq x$, $y=x\vee a$.
\end{center}
\end{definition}

A geometric lattice $\L$ is necessarily ranked and furthermore
it is upper semimodular, meaning that for any $x,y\in\L$:
\begin{equation*}
\rk(x) + \rk(y) \geq \rk(x\vee y) + \rk(x\wedge y).
\end{equation*}

\begin{definition}
A graded, bounded below poset  $\P$ is \defh{locally geometric} if, for every $x\in\P$, the subposet 
$\P_{\leq x}$ is a geometric lattice. 
\end{definition}

\begin{remark}
We do not require $\P$ itself to even be a (semi)lattice. If $\P$ is a lattice, then it is locally geometric if and only if
it is geometric.
\end{remark}

\begin{example}
A classical example of a geometric lattice is a Boolean lattice $B_n$, the set of all subsets of $[n]=\{1,2,\dots,n\}$ ordered by inclusion. A simplicial poset, in which every closed interval is isomorphic to a Boolean lattice, is then a locally geometric poset. The Hasse diagram of one such example is depicted in \Cref{fig:locgeom}; observe that this is a locally geometric poset that is not a lattice nor a semilattice.
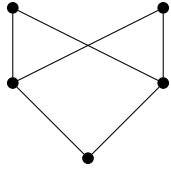
\begin{figure}[ht]
\begin{tikzpicture}
\solidnodes
\node (0) at (0,0) {};
\node (1) at (-1,1) {};
\node (2) at (1,1) {};
\node (a) at (-1,2) {};
\node (b) at (1,2) {};
\draw[-] (2)--(0)--(1)--(a)--(2)--(b)--(1);
\end{tikzpicture}
\caption{Hasse diagram of a locally geometric poset.}
\label{fig:locgeom}
\end{figure}
\end{example}

\begin{remark}\label{rem:subgl}
Let $\P$ be a locally geometric poset, and suppose that $x,y\in\P$ are such that
$x\vee y$ is nonempty. Then 
\begin{enumerate}
\item $x$ and $y$ have a (unique) greatest lower bound $x\wedge y$, and
\item for any $z\in x\vee y$, 
 $\rk(x)+\rk(y) \geq \rk(z)+\rk(x\wedge y)$.
\end{enumerate}
The reason is that both properties must hold in the
subposet $\P_{\leq z}$ whenever $z\in x\vee y$.
\end{remark}

\begin{remark}
If $\P$ is locally geometric, then so are $\P_{\leq x}$ and $\P_{\geq x}$ for
any $x\in\P$. 
\end{remark}

\subsection{Supersolvable geometric lattices}

There are several equivalent definitions for a modular element in a geometric lattice (see eg. \cite[Theorem 3.3]{brylawski}), and the following, due to Stanley \cite{stanley-modular}, is most useful for us. For this we need some more terminology. Let $\L$ be a chain-finite lattice. Then $\L$ has a unique minimal element $\zero$ as well as a unique maximal element, which we denote by $\one$. Let $x\in \L$. The \textit{complements} of $x$ in $\L$ are the elements $y\in\L$ such that $x\wedge y=\zero$ and $x\vee y=\one$.

\begin{definition}\label{def:modular}
An element $x$ in a geometric lattice $\L$ is \defh{modular} if 
the complements of $x$ form an antichain.
\end{definition}

\begin{remark}\label{rem:modular-equivalent} The following are equivalent for an element $x$ of a geometric lattice $\L$:
\begin{enumerate}
    \item \label{def:eqmod:0} $x$ is modular in $\L$;
   \item\label{def:eqmod:1}
   $\rk(x)+\rk(y)=\rk(x\vee y)+\rk(x\wedge y)$ for all $y\in\L$;
   \item\label{def:eqmod:2}
   $(u \vee y ) \wedge x = u\vee (y\wedge x)$ for all $u\leq x$ and all $y\in \L$.
\end{enumerate}   
All equivalences are well-known in the finite case. The proof of \eqref{def:eqmod:0}$\Leftrightarrow$\eqref{def:eqmod:1} given in \cite[Theorem 1]{stanley-modular} for the finite case carries over to the chain-finite setting using \cite[Theorem 2.29 and 6.4.(iii)]{aigner}. The equivalence \eqref{def:eqmod:0}$\Leftrightarrow$\eqref{def:eqmod:2} is proved in the chain-finite setting in \cite[Proposition 2.8]{crapo-rota}.
\end{remark}

\begin{example}
Let $\L=\Pi_4$ be the set of partitions of $\{1,2,3,4\}$ partially ordered by refinement, whose Hasse diagram is depicted in \Cref{fig:Pi4}.
The partition $123|4$ is modular because its set of complements $\{14|2|3, 24|1|3, 34|1|2\}$ is an antichain.

The partition $12|34$ is not modular because both $13|24$ and $13|2|4$ are complements of $12|34$ while $13|2|4<13|24$. One can alternatively see that $12|34$ is not modular from the inequality $\rk(12|34)+\rk(13|24) = 4 > 3 = \rk(1234) + \rk(1|2|3|4)$. 
\begin{figure}[ht]
{
\begin{tikzpicture}[scale=.7]
\node (0) at (0,0) {$1|2|3|4$};
\node (12) at (-5,2) {$12|3|4$};
\node (13) at (-3,2) {$13|2|4$};
\node (14) at (-1,2) {$14|2|3$};
\node (23) at (1,2) {$23|1|4$};
\node (24) at (3,2) {$24|1|3$};
\node (34) at (5,2) {$34|1|2$};
\node (123) at (-6,5) {$123|4$};
\node (124) at (-4,5) {$124|3$};
\node (a) at (-2,5) {$12|34$};
\node (c) at (0,5) {$14|23$};
\node (b) at (2,5) {$13|24$};
\node (134) at (4,5) {$134|2$};
\node (234) at (6,5) {$234|1$};
\node (1234) at (0,7) {$1234$};
\foreach \x in {12,13,14,23,24,34} {\draw[-] (\x) -- (0);};
\foreach \x in {123,124,134,234,a,b,c} {\draw[-] (\x) -- (1234);};
\foreach \x in {123,124,a} {\draw[-] (\x.south) -- (12.north);};
\foreach \x in {123,134,b} {\draw[-] (\x.south) -- (13.north);};
\foreach \x in {124,134,c} {\draw[-] (\x.south) -- (14.north);};
\foreach \x in {123,234,c} {\draw[-] (\x.south) -- (23.north);};
\foreach \x in {124,234,b} {\draw[-] (\x.south) -- (24.north);};
\foreach \x in {134,234,a} {\draw[-] (\x.south) -- (34.north);};
\end{tikzpicture}
}
\caption{The lattice $\Pi_4$ of partitions of $\{1,2,3,4\}$.}
\label{fig:Pi4}
\end{figure}
\end{example}

The following definition extends to the chain-finite case Stanley's criterion for when a finite geometric lattice is supersolvable \cite[Corollary 2.3]{stanley}. We will later further extend it to locally geometric posets.

\begin{definition}\label{def:sslattice}
A geometric lattice $\L$ is \defhno{supersolvable} if there is a chain $\zero=y_0<y_1<\cdots<y_n=\one$ where each $y_i$ is a modular element with $\rk(y_i)=i$.
\end{definition}

\begin{example}
In a Boolean lattice, the least upper bound of two subsets is their union while the greatest lower bound is their intersection.
It is easy to see that every element is a modular element, which implies that a Boolean lattice is supersolvable.
\end{example}

\begin{example}\label{ex:partitions}
The partition lattice $\Pi_n$ is 
the collection of set partitions of $[n]=\{1,2,\dots,n\}$ ordered by refinement. 
The modular elements of $\Pi_n$ correspond to partitions with at most one nonsingleton block, and one can build a chain of these elements to see that $\Pi_n$ is supersolvable.
\end{example}

\subsection{M-ideals}

Let $\P$ be a locally geometric poset. An \emph{order ideal} in $\P$ is a
downward-closed subset. An order ideal is \emph{pure} if all maximal elements
have the same rank. An order ideal $\Q $ is \emph{join-closed} if $T\subseteq \Q $
implies $\bigvee T\subseteq \Q $.

Here we introduce M-ideals to generalize the notion of modular elements beyond lattices. Our perspective is to rather generalize the order ideal generated by a modular element and how this ideal interacts with the entire poset. In the lattice case this ideal is principal and is therefore determined by its unique maximal element; in general this will not be the case.
The motivation for our definition is geometric (see \Cref{thm:fib}).

\begin{definition}\label{def:modideal}
An \defh{M-ideal} of a locally geometric poset $\P$ is a pure,
join-closed order ideal $\Q \subseteq\P$ such that:
\begin{deflist}
\item\label[definition]{def:mod:1}
if $y\in \Q $ and $a\in A(\P)$ such that $a\vee y=\emptyset$ then $a\in \Q $, and
\item\label[definition]{def:mod:2}
for every $x\in\max(\P)$, there is some $y\in\max(\Q )$ such that $y$ is a modular
element in the geometric lattice $\P_{\leq x}$.
\end{deflist}
\end{definition}

\begin{remark}\label{rem:uniquey}
Let $\Q\subseteq\P$ be an M-ideal, and let $x\in\max(\P)$. 
Since $\Q$ is join-closed, the $y\in\max(\Q)$ which is modular in $\P_{\leq x}$, guaranteed by \Cref{def:mod:2}, is necessarily unique.
\end{remark}

\begin{example}
In every locally geometric poset $\P$, both $\P$ and $\{\zero\}$ are M-ideals.
\end{example}

\begin{example}\label{ex:ss:mod}
Consider the poset $\P$ in \Cref{fig:ex:ss:mod}.
Both $\P_{\leq 1}$ and $\P_{\leq 3}$ are M-ideals in $\P$.
On the other hand, $\P_{\leq 2}$ is not an M-ideal, since $\max(\P_{\leq
2})=\{2\}$, $\max(\P)=\{a,b\}$, and $2\notin\P_{\leq b}$.
\begin{figure}[hb]
\begin{tikzpicture}
\node (T) at (0,0) {$\zero$};
\node (1) at (-1,1.5) {$1$};
\node (2) at (0,1.5) {$2$};
\node (3) at (1,1.5) {$3$};
\node (11) at (-1,3) {$a$};
\node (-11) at (1,3) {$b$};
\draw[-] (T)--(1);
\draw[-] (T)--(2);
\draw[-] (T)--(3);
\draw[-] (1)--(11);
\draw[-] (1)--(-11);
\draw[-] (2)--(11);
\draw[-] (3)--(11);
\draw[-] (3)--(-11);
\solidnodes
\end{tikzpicture}
\caption{In the poset $\P$ depicted here, the ideals $\P_{\leq 1}$ and $\P_{\leq 3}$ are M-ideals, while $\P_{\leq 2}$ is not (see \Cref{ex:ss:mod}).}
\label{fig:ex:ss:mod}
\end{figure}
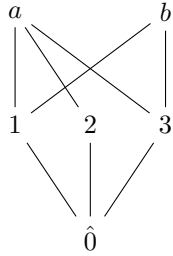
\end{example}

The following lemma shows that our definition of an M-ideal extends the definition of a modular element in a geometric lattice.
\begin{lemma}\label{lem:latticemodid}
An order ideal $\Q $ in a geometric lattice $\L$ is an M-ideal if and only if
$\Q =\L_{\leq y}$ for some modular element $y$.
\end{lemma}
\begin{proof}
Let $\L$ be a geometric lattice. Since joins are always nonempty in a lattice,
an order ideal $\Q $ is join-closed if and only if $\Q =\L_{\leq y}$ for some $y$.
Any such ideal is pure and satisfies condition \Cref{def:mod:1}. Since there is
a unique maximum in the lattice $\L$, condition \Cref{def:mod:2} is equivalent
to requiring that $y$ be modular in $\L$. 
\end{proof}

We conclude this subsection with several properties of M-ideals to be used later in this paper.
\begin{lemma}\label{lem:Qcomp}
Let $\Q $ be an M-ideal in a locally geometric poset $\P$ with
$\rk(\Q )=\rk(\P)-1$, and let $a\in\P$. Then $a\in A(\P)\setminus A(\Q )$ if and
only if $a\wedge y=\zero$ for all $y\in\max(\Q )$.
\end{lemma}
\begin{proof}
First, suppose that $a\in A(\P)\setminus A(\Q )$, and let $y\in\max(\Q )$. If
$a\wedge y\neq\zero$, then $a\wedge y = a$ since $a\in A(\P)$. However, this
implies $a\leq y$, and thus $a\in \Q $ because $\Q $ is an order ideal. This is a
contradiction; therefore $a\wedge y=\zero$.

Now suppose that $a\wedge y=\zero$ for all $y\in\max(\Q )$. This means, in
particular, that $a\not\leq y$ for any $y\in\max(\Q )$. Thus $a\notin \Q $, and we
need to show that $a\in A(\P)$. Now let $x\in\max(\P)$ be such that $a\leq x$.
By \Cref{def:mod:2}, there is some $y\in\max(\Q )$ such that $y\leq x$ and $y$ is
modular in $\P_{\leq x}$. Since $a\not\leq y$ and $\rk(y)=\rk(x)-1$, we must
have $a\wedge y=\zero$ and $a\vee y=x$ in $\P_{\leq x}$. By modularity of $y$ in
$\P_{\leq x}$, this implies $\rk(a)=\rk(x)+\rk(\zero)-\rk(y)=1$. Thus, $a\in
A(\P)$.
\end{proof}

\begin{proposition}\label{prop:localmod}
If $\Q $ is an M-ideal of a locally geometric poset $\P$ with $\rk(\Q )=\rk(\P)-1$, then for any
$x\in\P\setminus \Q $, there is some $y\in \Q $ such that $x$ covers $y$ and $y$ is
modular in $\P_{\leq x}$.
\end{proposition}
\begin{proof}
Let $\hat{x}\in\max(\P)$ be such that $\hat{x}\geq x$. Let
$\hat{y}\in\max(\Q )$ be such that $\hat{x}$ covers $\hat{y}$, guaranteed by \Cref{def:mod:2}. 
We have $\hat{x}\in\hat{y}\vee x$, because $x\notin Q$ implies $x\not\leq \hat{y}$. Let $y:=\hat{y}\wedge x$.
Via \Cref{rem:modular-equivalent}.\eqref{def:eqmod:1}, 
modularity of $\hat{y}$ in $\P_{\leq \hat{x}}$ implies $\rk(\hat{y})+\rk(x)=\rk(y)+\rk(\hat{x})$ and since $\rk(\hat{x})-\rk(\hat{y})=1$, this shows that $x\gtrdot y$.
In order to show that $y$ is modular in $\P_{\leq x}$, we check 
the condition in \Cref{rem:modular-equivalent}.\eqref{def:eqmod:2}. Let $u\leq y$ and let $z\in\P_{\leq x}$. Then in the lattice $\P_{\leq x}$
\begin{align*}
u\vee(z\wedge y)
=u\vee(z\wedge (x\wedge \hat{y}))
&=u\vee(z\wedge \hat{y})\\
&=(u\vee z)\wedge \hat{y}\\
&=(u\vee z) \wedge (x\wedge \hat{y})=
(u\vee z) \wedge y,
\end{align*}
where the middle equality holds by modularity of $\hat{y}$ in $\P_{\leq \hat{x}}$ (note that $u\leq \hat{y}$), the second and fourth equalities hold because $z\leq x$ and $u\vee z\leq x$, while  the first and last equalities are by definition of $y$.
\end{proof}

\begin{corollary}\label{cor:niceM}
Let $\Q $ be an M-ideal of a locally geometric poset $\P$ with $\rk(\Q )=\rk(\P)-1$. Then
\begin{itemize}
    \item[\mylabel{dagger}{$(\dagger)$}] for any two distinct
$a_1,a_2 \in A(\P)\setminus A(\Q)$ and every $x\in a_1\vee a_2$ there is $a_3\in A(\Q)$ with $x > a_3$.
\end{itemize}
\end{corollary}
\begin{proof} Note that $x\in \P\setminus \Q$ (otherwise the fact that $\Q$ is downward-closed would imply $a_1,a_2\in \Q$) and $\rk(x)=2$ by \Cref{rem:subgl}.(2). An application of \Cref{prop:localmod} to $x$ gives some $y\lessdot x$, $y\in \Q$. Since $\rk(x)=2$ and $\rk$ is a rank function, $\rk(y)=1$ and so $a_3:=y$ satisfies the claim.
\end{proof}

\begin{corollary}\label{cor:niceGL}
Let $\L$ be a geometric lattice, let $y\in \L$ with $y\neq\one$ and let $\Q:=\L_{\leq y}$.
Then $y$ is a modular element of $\L$ of rank $\rk(y)=\rk(\L)-1$ if and only if \ref{dagger} above holds for $\Q$.
\end{corollary}
\begin{proof}
By \Cref{lem:latticemodid}, $y$ being modular of rank $\rk(\L)-1$ is equivalent to $\Q$ being an M-ideal of rank $\rk(\L)-1$ and then \Cref{cor:niceM} implies that \ref{dagger} holds. On the other hand, assume \ref{dagger}.
We prove that the complements of $y$ in $\L$ all have rank one, implying both that $y$ is modular and that $\rk(y)=\rk(\L)-1$.
Let $z$ be a complement to $y$ in $\L$. Then $z>\zero$. 
If $\rk(z)>1$, there is some $z'\leq z$ with $\rk(z')=2$, and every atom below $z'$ is in $\L\setminus\Q$ (otherwise $z\wedge y>\zero$),
contradicting \ref{dagger}. Thus $\rk(z)=1$ as was to be shown. 
\end{proof}
\subsection{Supersolvability}
We are now prepared to present our definition of a supersolvable locally geometric poset, which extends the definition of a supersolvable geometric lattice (cf.\ \Cref{def:sslattice}).
\begin{definition}\label{def:ss}
A locally geometric poset $\P$ is \defh{supersolvable} if there is a chain
\[
\zero = \Q_0 \subset \Q_1 \subset \cdots \subset \Q_n = \P\]
 where each $\Q_i$ is an M-ideal of $\Q_{i+1}$ with $\rk(\Q_i)=i$.
\end{definition}

\begin{example}\label{ex:ss:ss}
Recall the poset $\P$ from \Cref{ex:ss:mod} (see also \Cref{fig:ex:ss:mod}).
It is supersolvable via the chain $\zero\subset\P_{\leq 1}\subset\P$.
\end{example}

\begin{proposition}\label{prop:ss2ways}
If $\L$ is a geometric lattice, then $\L$ satisfies \Cref{def:ss} if and only if it satisfies
\Cref{def:sslattice}.
\end{proposition}
\begin{proof}
Via \Cref{lem:latticemodid}, a geometric lattice $\L$ satisfies \Cref{def:ss} if and only if there is a chain
\[ \zero \subset \Q_{\leq y_1} \subset \cdots \subset \Q_{\leq y_n} = \L \]
with each $y_i$ a modular element of rank $i$. In particular, this is equivalent
to the existence of a maximal chain of modular elements $\zero<y_1<\cdots<y_n=\one$ 
as required by \Cref{def:sslattice}. 
\end{proof}

\begin{remark}
If a locally geometric poset is supersolvable, then every closed interval $\P_{\leq x}$ is a supersolvable
geometric lattice. However, this ``local'' supersolvability is not enough for $\P$
itself to be supersolvable, as demonstrated in the following example.
\end{remark}

\begin{example}\label{ex:locss}
Consider the poset $\P$ whose Hasse diagram is depicted in \Cref{fig:ex:locss},
which is the Boolean algebra on three generators with the maximum element
removed.
Every closed interval in $\P$ is supersolvable (since every Boolean lattice is), however it is not itself supersolvable.
Indeed, the only proper order ideals which are pure and join-closed are
principal, that is, $\P_{\leq x}$ for some rank-one element $x$. However, such
an order ideal cannot satisfy \Cref{def:mod:2} since no single element is
covered by all maximal elements.

This particular poset describes the intersection data of an affine hyperplane arrangement, explicitly the de-cone of an arrangement $\A$ of four generic hyperplanes in $\C^3$. The arrangement $\A$ is not supersolvable either; this is not a coincidence and will be made explicit in Theorem \ref{thm:sscone}.

\begin{figure}[ht]
\begin{tikzpicture}[scale=1.2]
\node (0) at (0,0) {$\zero$};
\node (1) at (-1.5,1.5) {$\{1\}$};
\node (2) at (0,1.5) {$\{2\}$};
\node (3) at (1.5,1.5) {$\{3\}$};
\node (a) at (-1.5,3) {$\{1,2\}$};
\node (b) at (0,3) {$\{1,3\}$};
\node (d) at (1.5,3) {$\{2,3\}$};
\foreach \x in {1,2,3} \draw[-] (0) -- (\x);
\foreach \x in {a,b} \draw[-] (1.north) -- (\x.south);
\foreach \x in {a,d} \draw[-] (2.north) -- (\x.south); 
\foreach \x in {b,d} \draw[-] (3.north) -- (\x.south); 
\end{tikzpicture}
\caption{
This locally geometric poset is ``locally'' supersolvable but not 
supersolvable (see \Cref{ex:locss}).}
\label{fig:ex:locss}
\end{figure}
\end{example}

\subsection{Dowling posets}

Dowling posets \cite{BG} form a class of locally geometric posets which motivates
the general definition of supersolvability.
These generalize partition lattices and Dowling lattices, which are
known to be supersolvable geometric lattices \cite{stanley,dowling}.
To define these posets, let us fix a positive integer $n$, a finite group $G$,
and a finite $G$-set $S$. Denote $[n]=\{1,2,\dots,n\}$. 

Given a subset $B\subseteq[n]$, a $G$-coloring is a function $b:B\to G$. Define
an equivalence relation on $G$-colorings of $B$ where $(b:B\to G)\sim(b':B\to
G)$ whenever $b'=bg$ for some $g\in G$. 
Note that if $B=\{k\}$ is a singleton then all $G$-colorings are equivalent.
A partial $G$-partition of $[n]$ is a
collection $\beta=\{(B_1,\overline{b_1}),\dots,(B_\ell,\overline{b_\ell})\}$
where $\{B_1,\dots,B_\ell\}$ is a partition of some subset $T\subseteq [n]$ and
each $\overline{b_i}$ is a chosen equivalence class of $G$-colorings on $B_i$.
Given such a partial $G$-partition, denote $Z_\beta=[n]\setminus\cup_i B_i$.

Let $D_n(G,S)$ be the set of pairs $(\beta,z)$ where $\beta$ is a partial
$G$-partition of $[n]$ and $z:Z_\beta\to S$.
This set is partially ordered via the covering relations:
\begin{itemize}
\item $(\beta\cup\{(A,\overline{a}),(B,\overline{b})\},z) \prec
(\beta\cup\{(A\cup B,\overline{a\cup bg})\},z)$ whenever $g\in G$, and
\item $(\beta\cup\{(B,\overline{b})\},z)\prec(\beta,z')$ whenever $z':B\cup
Z_\beta\to S$ satisfies $z'|_{Z_\beta}=z$ and $z'|_B=f\circ b$ for some
$G$-equivariant function $f:G\to S$.
\end{itemize}
As shown in \cite[Theorem A]{BG}, this is a locally geometric poset whose
maximal intervals are products of partition and Dowling lattices.

\begin{proposition}\label{prop:dowling}
For any positive integer $n$, finite group $G$, and finite $G$-set $S$, the
Dowling poset $D_n(G,S)$ is supersolvable.
\end{proposition}
\begin{proof}
We proceed by induction on $n$. The case $n=1$ is immediate, so let $n>1$.
There is an injective map of posets $D_{n-1}(G,S)\to D_n(G,S)$ defined by
$\iota(\beta,z) = (\beta\cup\{n\},z)$.
The image $\Q :=\im(\iota)$ is a pure, join-closed order ideal of $D_n(G,S)$,
isomorphic to $D_{n-1}(G,S)$, and so by induction it suffices to prove that $\Q $
satisfies the conditions of \Cref{def:mod:1,def:mod:2}.

Using the description of atoms in $D_n(G,S)$ from \cite[Lemma 2.5.1]{BG},
we see that the atoms which are not in $\Q$ are of the following form:
\begin{itemize}
\item a $G$-partition of $[n]$ whose only nonzero singleton block is $\{i,n\}$,
for some $i<n$, or
\item the partition of $[n-1]$ into singletons along with any function
$z:\{n\}\to S$.
\end{itemize}
In each case, it is straightforward to check that such an atom has a unique
minimal upper bound with an element of the form $(\beta\cup\{n\},z)$.

Finally, to see \Cref{def:mod:2}, we consider two cases: either $S$ is empty or
nonempty. If $S$ is nonempty, then the maximal elements of $D_n(G,S)$ are
of the form $(\emptyset,z)$, and the interval $D_n(G,S)_{\leq(\emptyset,z)}$ is
isomorphic to the product of Dowling lattices $D_{z^{-1}(Gs)}(G_s)$ where $Gs\in
S/G$. In particular, the element $(\{n\},z|_{Z-\{n\}})$ is a modular element of
this interval and also a maximal element of $\Q $. 
If $S$ is empty, then every maximal interval is isomorphic to a partition
lattice $\Pi_n$, of which a corank-one element with singleton $\{n\}$ is
modular.
\end{proof}

These Dowling posets were defined to describe the intersection data of a collection of submanifolds whose complement is an orbit configuration space -- an equivariant analogue of configuration spaces first studied by Xicot\'encatl \cite{X}. 
The supersolvability of these posets, and Xicot\'encatl's study of fiber bundles on orbit configurations, 
hints at a much larger phenomenon that we study in the next section.

\section{Topological fibrations of arrangements}\label{sec:fib}

In this section, we establish a topological interpretation of an M-ideal, generalizing Terao's Fibration Theorem \cite{teraofibthm} on hyperplane arrangements. First, we fix notation and terminology to be used throughout.

\subsection{Arrangements}\label{sec:def:alga}
Fix a finite-dimensional connected
abelian Lie group $\G$, so that $\G\cong(S^1)^d\times\R^v$ for some nonnegative
integers $d,v$. Also fix a finite-rank free abelian group $\Gamma$ and
$T=\Hom(\Gamma,\G)$. Note that the group operation on $\G$ induces a group operation on $T$, which we denote by
$+$. For sets $U,V\subseteq T$, let 
$U+V := \{t +s\st t\in U,s\in V\}\subseteq T$.
If $U=\{t\}$, we abbreviate $t+V=U+V$, which is a translation of $V$ by $t$.

\begin{definition}
\label{def:alga}
An \defh{abelian Lie group arrangement}, or an \emph{abelian arrangement} for short, is a collection $\{H_\alpha\st
\alpha\in\A\}$ for some finite set $\A\subseteq\Gamma$, where
\[H_\alpha := \{t\in T\st t(\alpha)=0\}.\]
The complement of $\A$ is denoted by
\[ M(\A) := 
T\setminus \bigcup_{\alpha\in\A} H_\alpha.\]

\textit{Linear, toric,} and \textit{elliptic} arrangements are abelian Lie group arrangements with $\G = \C$, $\mathbb C^\times$ or a complex elliptic curve, respectively (here $\G\cong(S^1)^d\times\R^v$ for $(d,v)=(0,2)$,
$(1,1)$, and $(2,0)$).

We will often refer to an arrangement $\{H_\alpha\st \alpha\in\A\}$ simply by $\A$ when there is no confusion.

\end{definition}

\begin{definition}\label{def:layers}
A \defh{layer} of an arrangement $\A$ is a nonempty connected component of an intersection $\bigcap_{\alpha\in S}H_\alpha$ where $S\subseteq\A$.
The set $\P(\A)$ of layers, partially ordered by reverse inclusion, is called the \defh{poset of layers}.
\end{definition}

\begin{example}[Graphic arrangements and configuration spaces]\label{ex:graphic}
Every finite simple graph determines an abelian Lie group arrangement in the following way.
Let $G$ be a finite simple graph with vertex set $[n]=\{1,2,\dots,n\}$ and edge set $E$. Let $\Gamma$ be a free abelian group with basis  $\beta_1,\dots,\beta_n$. Given two elements $1\leq i<j\leq n$, define $\alpha_{i,j}=\beta_i-\beta_j \in\Gamma$ and abbreviate $H_{i,j}:=H_{\alpha_{i,j}}\subseteq\Hom(\Gamma,\G)$. 
Now let $\A_G := \{H_{i,j}\st \{i,j\}\in E\}$, the arrangement associated to the graph $G$.
In the case that $G=K_n$ is the complete graph on $n$ vertices, the complement $M(\A_{K_n})$ is the configuration space of $n$-tuples of distinct points in $\G$, denoted by $\Conf_n(\G)$. For an arbitrary simple graph, the complement is sometimes called a \textit{partial configuration space}, allowing points to collide.

The poset of layers of $\A_{K_n}$ is isomorphic to the partition lattice $\Pi_n$, which we saw in \Cref{ex:partitions} was supersolvable.
More generally, the poset of layers for a graphic arrangement $\A_G$ is supersolvable if and only if the graph $G$ is chordal \cite{stanley}.
\end{example}

By convention, $T$ is the unique minimum element of $\P(\A)$, thought of as the
empty intersection. 
The poset of layers for an abelian arrangement can be realized as the quotient of a geometric semilattice by a translative group action (see \cite[Lemma 9.8]{dD} and \Cref{def:translative} below) and hence $\P(\A)$ is a locally geometric poset. 
The atoms of $\P(\A)$ are precisely the connected components
of the $H_\alpha$, where $\alpha\in\A$.
Note that if $K$ is the connected component of an intersection
$X=\cap_{\alpha\in S}H_\alpha$ passing through the identity of $T$, then every
connected component of $X$ is of the form $t+K$ for some $t\in T$.

\begin{d-a}
We call an arrangement \defh{essential} if $\A$ generates a full subgroup of $\Gamma$. The arrangement is \defh{irredundant} if, for distinct $\alpha$ and $\beta$, $H_\alpha$ and $H_\beta$ do not share a connected component. All arrangements that we will consider will be essential and irredundant.
\end{d-a}

\begin{example}\label{ex:ss}
Let $\Gamma=\mathbb Z^2$ and $\A=\{\alpha_1=(1,0),\alpha_2=(0,1),\alpha_3=(1,2)\}$. Let $\G=S^1\times \mathbb R^v$ and
consider an arrangement in $T$. If $v=0$ or $v=1$, we may identify
$T$ with $(S^1)^2$ or $(\C^\times)^2$, respectively.
\Cref{fig:ex:ss} depicts the arrangement in $(S^1)^2$ and the Hasse diagram for
its poset of layers in either case.
\begin{figure}[ht]
\begin{tikzpicture}[scale=2]
\draw[thick,-] (0,0)--(1,0)--(1,1)--(0,1)--(0,0);
\draw[thick,red,-] (0,0)--(1,0);
\draw[thick,green,-] (0,0)--(0,1);
\draw[thick,blue,-] (0,0)--(1,.5);
\draw[thick,blue,-] (0,.5)--(1,1);
\node at (0,-.25) {};
\solidnodes
\node at (0,0) {};
\node at (0,.5) {};
\end{tikzpicture}
\hspace{1in}
\begin{tikzpicture}
\node (T) at (0,0) {$T$};
\node (1) at (-1,1.5) {$H_1$};
\node (2) at (0,1.5) {$H_2$};
\node (3) at (1,1.5) {$H_3$};
\node (11) at (-1,3) {$(1,1)$};
\node (-11) at (1,3) {$(1,-1)$};
\draw[-] (T)--(1);
\draw[-] (T)--(2);
\draw[-] (T)--(3);
\draw[-] (1)--(11);
\draw[-] (1)--(-11);
\draw[-] (2)--(11);
\draw[-] (3)--(11);
\draw[-] (3)--(-11);
\solidnodes
\end{tikzpicture}
\caption{The arrangement $\A$ from Example \ref{ex:ss} is depicted on the left,
with $H_1$ in \textcolor{green}{green}, $H_2$ in \textcolor{red}{red}, and $H_3$
in \textcolor{blue}{blue}. Its poset of layers $\P(\A)$ is depicted on the right.}
\label{fig:ex:ss}
\end{figure}
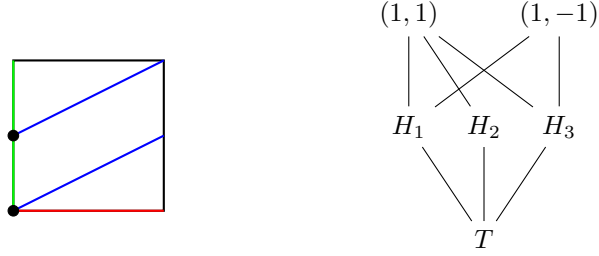

\end{example}

\begin{definition}
\label{def:Y} 
A subgroup $Y$ of $T$ will be called \defh{admissible} if there is a direct sum decomposition $\Gamma = \Gamma' \oplus \Gamma''$ such that $\Gamma'$ has rank $1$ and
$Y$ is the image of the injection $\varepsilon^*:\Hom(\Gamma',\G)\to\Hom(\Gamma,\G)$ induced by the projection $\varepsilon:\Gamma\to\Gamma'$. 
Choose a generator $\Gamma'=\langle\beta_0\rangle$, and define for $\alpha\in\Gamma$ a nonnegative integer $c(\alpha)$ such that $\varepsilon(\alpha)=\pm c(\alpha)\beta_0$. 
\end{definition}

If $Y$ is admissible, the corresponding projection \[p:T\to
T/Y\cong\Hom(\Gamma/\Gamma',\G)\] is a
section of the map induced by the quotient \[q:\Gamma\to\Gamma/\Gamma'.\]

This allows us to define the sub-arrangement
\[\A_Y := \{\alpha\in\A\st H_\alpha\supseteq Y\} = 
\{\alpha\in\A \st c(\alpha)=0\}\]

\begin{remark}\label{rem:aay}
The set of atoms $A(\A_Y)$ consists of all connected components of the $H_\alpha$ with $\alpha\in \A_Y$. These are the atoms of $\A$ that either contain $Y$ or are disjoint from it. 
For any $\alpha\not\in\A_Y$, every connected component of $H_\alpha$ will intersect $Y$ nontrivially. In particular, the poset of layers $\P(\A_Y)$ may be viewed as a subposet of $\P(\A)$. Moreover, if $Y\in \P(\A)$, then the maximal elements of $\P(\A_Y)$ are cosets of $Y$.
\end{remark}

 The set
\[\qA := q(\A_Y)\subseteq\Gamma/\Gamma'\]
defines an arrangement in
$T/Y$. The map $p:T\to T/Y$ restricts to a map on arrangement complements
$\pr:M(\A)\to M(\qA)$ and induces an isomorphism of posets \[\P(\A_Y)\cong\P(\qA).\]

Following the terminology of Terao \cite{teraofibthm}, associated to the arrangement $\A$ and the projection $p$ we define the {\em horizontal set} by 
\[\Hor := \{X\in\P_{>T} \st p(X) = T/Y\},\]
the {\em bad set} by
\[\Bad := \bigcup_{\substack{X\in\P_{>T}\\ X\notin\Hor}} p(X)\cap M(\qA),\]
and for $t\in T$,
\[\P_t := \{X\in\P_{>T}\st(t+Y)\cap X\neq\emptyset\}.\]

\begin{example}\label{ex:fib:good}
Consider the arrangement from \Cref{ex:ss} (see also \Cref{fig:ex:ss}) with $\Gamma=\Z^2$, 
$\A=\{\alpha_1=(1,0),\,\alpha_2=(0,1),\,\alpha_3=(1,2)\}$,
and $\G=S^1\times\R^v$.
Abbreviate $H_i=H_{\alpha_i}$ for $i=1,2,3$, and let $Y=H_1$.
Then $\A_Y=\{\alpha_1\}$ and the projection $M(\A)\to M(\qA)$ is depicted in
Figure \ref{fig:fib:good}.

As the picture suggests, this map is a fibration with fiber homeomorphic to $T$
with three points removed.
In this case, the bad set is empty ($\Bad=\emptyset$) while the horizontal set
and $\P_t$ (for $t\in M(\A)$) are both equal to $\{H_2,H_3\}$.

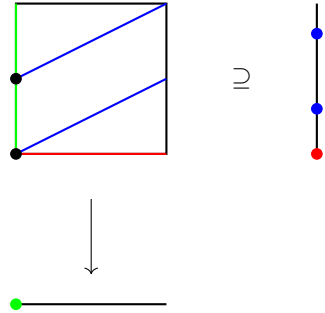
\begin{figure}[ht]
\begin{tikzpicture}[scale=2]
\draw[->] (.5,-.3)--(.5,-.8);
\node at (1.5,.5) {$\supseteq$};
\draw[thick,-] (0,0)--(1,0)--(1,1)--(0,1)--(0,0);
\draw[thick,red,-] (0,0)--(1,0);
\draw[thick,green,-] (0,0)--(0,1);
\draw[thick,blue,-] (0,0)--(1,.5);
\draw[thick,blue,-] (0,.5)--(1,1);
\node at (0,-.25) {};
\node[green] at (-.15,.75) {\scriptsize $H_1$};
\node[blue] at (.35,.35) {\scriptsize $H_3$};
\node[red] at (.75,-.15) {\scriptsize $H_2$};
\solidnodes
\node at (0,0) {};
\node at (0,.5) {};
\draw[thick,-] (0,-1) -- (1,-1);
\node[green,fill=green] at (0,-1) {};
\draw[thick,-] (2,0)--(2,1);
\node[red,fill=red] at (2,0) {};
\node[blue,fill=blue] at (2,.3) {};
\node[blue,fill=blue] at (2,.8) {};
\end{tikzpicture}
\caption{The restriction of the projection $S^1\times S^1\to S^1$ to the
complement of the arrangement $\A$ from \Cref{ex:fib:good} is a fibration
whose fibers are homeomorphic to the circle $S^1$ with three punctures.}
\label{fig:fib:good}
\end{figure}
\end{example}

\begin{example}\label{ex:fib:bad}
Consider again the arrangement $\A$ from \Cref{ex:fib:good}, but take $Y=H_2$.
The projection $M(\A)\to M(\qA)$ is depicted in Figure \ref{fig:fib:bad}, from
which it is evident that this map is not a fibration. Indeed, the fiber over a
point $t\in M(\qA)$ is homeomorphic to $T$ with two punctures, except the one
case that $t=-1$ and the fiber is homeomorphic to $T$ with a single puncture.

In this example, the bad set is nonempty: $\Bad=\{-1\}$, the horizontal set is
$\Hor=\{H_1,H_3\}$, and the set $\P_t$ is equal to $\{H_1,H_3\}$ for all $t\in
M(\qA)$ except $t=-1$, for which we have $\P_{-1}=\{H_1,H_3,(1,-1)\}$.

\begin{figure}[ht]
\begin{tikzpicture}[scale=2]
\foreach \x in {0.5,2,2.5} \draw[->] (\x,-.3)--(\x,-.8);
\node at (1.5,.5) {$\supseteq$};
\node at (2,-1) {\scriptsize $-1$};
\node at (2.5,-1) {\scriptsize $t\neq -1$};
\draw[thick,-] (0,0)--(1,0)--(1,1)--(0,1)--(0,0);
\draw[thick,red,-] (0,0)--(0,1);
\draw[thick,green,-] (0,0)--(1,0);
\draw[thick,blue,-] (0,0)--(.5,1);
\draw[thick,blue,-] (.5,0)--(1,1);
\node at (0,-.25) {};
\solidnodes
\node at (0,0) {};
\node at (.5,0) {};
\draw[thick,-] (0,-1) -- (1,-1);
\node[red,fill=red] at (0,-1) {};
\draw[thick,-] (2,0)--(2,1);
\draw[thick,-] (2.5,0)--(2.5,1);
\node[cyan,fill=cyan] at (2,0) {};
\node[green,fill=green] at (2.5,0) {};
\node[blue,fill=blue] at (2.5,.3) {};
\end{tikzpicture}
\caption{The restriction of the projection $S^1\times S^1\to S^1$ (left) is not a fibration, as indicated by the two non-homeomorphic fibers
presented (right). See \Cref{ex:fib:bad}.}
\label{fig:fib:bad}
\end{figure}
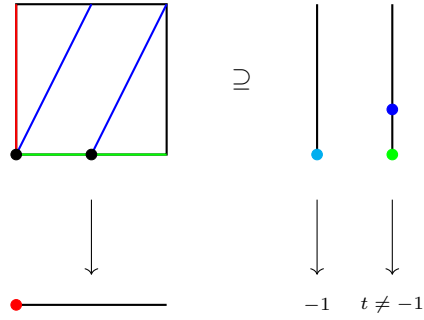
\end{example}

\subsection{Characterization of fibrations}
In this section, we describe conditions on an admissible subgroup $Y$ which will imply that the map $\pr:M(\A)\to M(\qA)$ is a fiber bundle (see \Cref{prop:Ygood}).

\begin{lemma}\label{lem:calpha}
If $\alpha\in \A\setminus \A_Y$, then for any $t\in T$ the intersection $(t+Y)\cap H_\alpha$ is a set of $c(\alpha)^d$ points. 
If $\alpha\in\A_Y$, then for any $t\notin H_\alpha$ the intersection $(t+Y)\cap H_\alpha$ is empty.
\end{lemma}
\begin{proof} Since $\Gamma'$ is a direct summand of $\Gamma$ we have a direct sum decomposition $\Gamma=\Gamma'\oplus\Gamma''$, with $\Gamma''=\Gamma/\Gamma'$. Recall from \Cref{def:Y} our choice of  $\beta_0$ such that $\Gamma'=\angles{\beta_0}$. Without loss of generality suppose that for our $\alpha\in \Gamma$ the nonnegative integer $c(\alpha)$ satisfies $$\alpha=(c(\alpha)\beta_0) + \alpha''$$ with $\alpha''\in \Gamma''$.

The layer $Y\subseteq T$ consists of all homomorphisms $x\in T=\Hom(\Gamma,\G)$ with $x(\Gamma'')=0$. Thus, 
$$Y\cap H_\alpha =\{x\in T : x(\Gamma'')=0 \textrm{ and } c(\alpha)x(\beta_0)=0\}$$ 
(the second condition is equivalent to $x(\alpha)=0$).
This intersection thus equals $Y$ when $\alpha\in \A_Y$ (i.e, when $c(\alpha)=0$). Otherwise, it is in bijection with the possible choices for $x(\beta_0)$, namely of the set of all $g\in \G$ with $c(\alpha)g=0$. This set has $c(\alpha)^d$ elements (as finite-order elements of $\G$ must lay in the factor $(S^1)^d$).

Now let $t\in T$ and consider
$$
(t+Y)\cap H_\alpha =
\{
t+x\in T : x(\Gamma'')=0\textrm{ and }
(t+x)(\alpha)=0
\}
$$
\def\lawG{+_{\G}}
\def\idG{0_\G}
The last condition is equivalent to $c(\alpha)(t(\beta_0)\lawG x(\beta_0) )\lawG t(\alpha'') = 0$ (here $\lawG$ denotes the group operation in $\G$). 
If $\alpha\in \A_Y$ then $c(\alpha)=0$ and the latter equation can only be satisfied if $t(\alpha'')=0$, i.e., if $t\in H_\alpha$. The second part of the claim follows.

If $\alpha\in \A\setminus\A_Y$, then $c(\alpha)\neq 0$ and so there is $g\in \G=(S^1)^d\times \R^v$ such that $c(\alpha)(t(\beta_0)\lawG g)=-t(\alpha'')$. Let $\overline{x}\in T$ be defined by setting $\overline{x}(\Gamma'')=\{0\}$ and  $\overline{x}(\beta_0)=g$. Then the assignment $Y\cap H_\alpha \to (t+Y)\cap H_\alpha$, $x\mapsto (t+x+\overline{x})$ is a bijection and shows that $(t+Y)\cap H_\alpha$ and $Y\cap H_\alpha$ have the same cardinality. 
\end{proof}

\begin{lemma}\label{lem:Hor}
The horizontal set with respect to the projection $p: T\to T/Y$ is
$$\Hor = A(\A\setminus\A_Y).$$
\end{lemma}
\begin{proof}
Let $X\in\P_{>T}$, and let $X_0=t+X$ be a translate of $X$ which is a (proper) subgroup of $T$. Since every $H_\alpha$ contains the identity, $X_0\in \P$.  Then 
\begin{align*}
X\in\Hor
\iff p(X) = T/Y &\iff p(X_0)=T/Y
\iff X_0+Y=T \\&\iff X_0\not\supseteq Y \text{ and } \rk(X_0)=1\\
&\iff X_0\in A(\A\setminus \A_Y)\iff X\in A(\A\setminus \A_Y),
\end{align*}
where the last claim holds because $X$ and $X_0$ are necessarily connected components of the same $H_\alpha$ (for some $\alpha\in\A\setminus\A_Y$).
\end{proof}

\begin{lemma}\label{lem:subsets}
For any $t\in M(\A)$,
\begin{align}
\Hor&\subseteq\P_t \label{eq:HorPt}\\
&\subseteq\{X\in\P\colon p(X)\cap M(\qA)\neq\emptyset\} \label{eq:PtBad}\\
&=\{X\in\P\colon X\wedge Y'=T \text{ for all }
Y'\in\tr{Y}\} \label{eq:BadMeet}
\end{align}
\end{lemma}
\begin{proof}
Combining \Cref{lem:calpha,lem:Hor}, we have
$\Hor=A(\A\setminus\A_Y)\subseteq\P_t$, establishing \eqref{eq:HorPt}.

For \eqref{eq:PtBad}, consider $X\in\P_t$ and let $u\in (t+Y)\cap X$. Then
$u\in X$ implies $p(u)\in p(X)$, and $u\in t+Y$ implies $p(u)=p(t)\in M(\qA)$.
Thus, $p(u)\in p(X)\cap M(\qA)$.

For \eqref{eq:BadMeet}, first assume that $X\in\P$ and $X\wedge Y'\neq T$ for
some $Y'\in\tr{Y}$.
Then there is some $H\in A(\P)$ such that $X\subseteq H$ and
$Y'\subseteq H$. Now, $X\subseteq H$ implies $p(X)\cap M(\qA)\subseteq
p(H)\cap M(\qA)$. Since $Y'\subseteq H$, we must have $p(H)\cap
M(\qA)=\emptyset$. Therefore, $p(X)\cap M(\qA)=\emptyset$.

Conversely, suppose that $p(X)\cap M(\qA)=\emptyset$. Then there exists $H\in
A(\A_Y)$ such that $p(X)\subseteq p(H)$. But then $X\subseteq H$ and
$Y'\subseteq H$ for some $Y'\in\tr{Y}$, implying that $X\wedge Y'\neq
T$.
\end{proof}

The reason the set $\Bad$ is named such is that points \emph{not} in this set
have ``good'' fibers, as we will see in the following lemma.
\begin{lemma}\label{lem:tgood}
For any $t\in M(\A)$, the following statements are equivalent.
\begin{enumerate}
\item\label{lem:tgood:Bad} $p(t)\notin\Bad$
\item\label{lem:tgood:Pt=Hor} $\P_t=\Hor$
\item\label{lem:tgood:Int} For any distinct $\alpha,\beta\in\A\setminus\A_Y$,
the intersection $(t+Y)\cap H_\alpha\cap H_\beta$ is empty.
\item\label{lem:tgood:fiber} The fiber $\pr^{-1}(p(t))$ is homeomorphic to $\G$
with $\sum_{\alpha\in\A\setminus\A_Y} c(\alpha)^d$ points removed.
\end{enumerate}
\end{lemma}

\begin{proof}\strut
\begin{description}
\imp{lem:tgood:Bad}{lem:tgood:Pt=Hor}
By \Cref{lem:subsets}, we need only show that $\P_t\subseteq\Hor$. Let
$X\in\P_t$, so that $(t+Y)\cap X\neq\emptyset$, and fix $u\in (t+Y)\cap X$.
Then since $u\in t+Y$, we have $p(u)=p(t)$. 
However, $u\in X$ implies $p(t)=p(u)\in p(X)$, and $t\in M(\A)$ implies $p(t)\in
M(\qA)$, so $p(t)\in p(X)\cap M(\qA)$. Since $p(t)\notin\Bad$, this implies
$X\in\Hor$.

\imp{lem:tgood:Pt=Hor}{lem:tgood:Bad}
Suppose that $p(t)\in\Bad$, and let $X\in\P_{>T}\setminus\Hor$ such that
$p(t)\in p(X)\cap M(\qA)$. Since $p(t)\in p(X)$, we have $u\in X$
such that $p(t)=p(u)$, hence $u=t+y$ for some $y\in Y$.
Thus, $(t+Y)\cap X\neq\emptyset$ and so $X\in\P_t$, contradicting $\P_t=\Hor$.

\imp{lem:tgood:Pt=Hor}{lem:tgood:Int}
If $(t+Y)\cap (H_\alpha\cap H_\beta)$ is nonempty, then there is a connected
component $X$ of $H_\alpha\cap H_\beta$ such that $X\in\P_t$. But if $\P_t=\Hor$
then \Cref{lem:Hor} implies that $X\in A(\A\setminus\A_Y)$. In particular, $X$ is a connected component of both $H_\alpha$ and $H_\beta$, contradicting
the assumption that $\A$ is irredundant.

\imp{lem:tgood:Int}{lem:tgood:Pt=Hor}
By \Cref{lem:subsets}, we need only show that $\P_t\subseteq\Hor$. 
Suppose not, and let $X\in\P_t\setminus\Hor$. 
By \Cref{lem:Hor}, we have
$\rk(X)>1$, so pick distinct $\alpha,\beta\in\A$ such that $H_\alpha\cap
H_\beta\supseteq X$. 
Then $(t+Y)\cap X\neq\emptyset$ implies that each of $(t+Y)\cap H_\alpha$, $(t+Y)\cap H_\beta$, and $(t+Y)\cap(H_\alpha\cap H_\alpha)$ is nonempty. The first two of these being nonempty implies $\alpha,\beta\in\A\setminus\A_Y$ via \Cref{lem:calpha}, which contradicts assumption \eqref{lem:tgood:Int}. 

\equ{lem:tgood:Int}{lem:tgood:fiber}
We have:
\begin{align*}
\pr^{-1}(p(t))
&= (t+Y)\cap M(\A)\\
&= (t+Y)\setminus \bigcup_{X\in\P_t} ((t+Y)\cap X)\\
&= (t+Y)\setminus \bigcup_{\alpha\in\A\setminus\A_Y} ((t+Y)\cap H_\alpha)
\\
&\cong \G\setminus\{k \text{ points}\},
 \text{ with } k = \left|\bigcup_{\alpha\in\A\setminus\A_Y} ((t+Y)\cap
H_\alpha)\right|
\end{align*}
where the third equality holds by \Cref{lem:Hor}  
and \Cref{lem:subsets}\eqref{eq:HorPt}. 
Then \Cref{lem:calpha} implies that $k=\sum_{\alpha\in\A\setminus\A_Y} c(\alpha)^d$ if and only if the sets
$(t+Y)\cap H_\alpha$ are pairwise disjoint, which is precisely condition \eqref{lem:tgood:Int}.
\end{description}
\end{proof}

When the statements in \Cref{lem:tgood} hold for all $t\in M(\A)$, we obtain a useful description of the horizontal set $\Hor$ and relate it back to M-ideals.
\begin{proposition}\label{prop:Ygood}
The following statements are equivalent.
\begin{enumerate}
\item\label{thm:Ygood:modular} 
$\P(\A_Y)$ is an M-ideal of $\P(\A)$ with $\rk(\P(\A_Y))=\rk(\P(\A))-1$.
\item\label{prop:Ygood:Hor} $\Hor = \{X\in\P\colon
X\wedge Y'=T \text{ for all } Y'\in\tr{Y}\}$.
\item\label{prop:Ygood:Pt=Hor} $\P_t=\Hor$ for any $t\in M(\A)$.
\item\label{prop:Ygood:Pt=Pu} $\P_t$ does not depend on $t\in M(\A)$.
\item\label{prop:Ygood:Bad} $\Bad=\emptyset$.
\end{enumerate}
\end{proposition}

\begin{proof}
\
\begin{description}

\imp{thm:Ygood:modular}{prop:Ygood:Hor}
This follows by \Cref{lem:Qcomp,lem:Hor}.

\imp{prop:Ygood:Hor}{thm:Ygood:modular}
It is clear that $\P(\A_Y)$ is a pure, join-closed, order ideal of $\P:=\P(\A)$.
\Cref{def:mod:1} follows from \Cref{lem:calpha}.

To show that \Cref{def:mod:2} holds, let $X\in\max(\P)$. Since $\P(\A_Y)$ is join-closed, the set $\P(\A_Y)\cap\P_{\leq X}$ is a sublattice of $\P_{\leq X}$ and thus has a unique maximum element $W$. 
Let $Z\in\P_{\leq X}$ be a complement of $W$, so that $W\wedge Z=T$ and $W\vee Z=X$ in $\P_{\leq X}$.
Let $Y'\in\tr{Y}=\max\P(\A_Y)$ and $U\in Y'\wedge Z$. 
Then $U\in\P(\A_Y)\cap\P_{\leq X}$ and hence $U\leq W$.
But $W\wedge Z=T$, so the fact that $U\leq W$ and $U\leq Z$ implies $U=T$. 
Therefore, $Y'\wedge Z=T$ for any $Y'\in\tr{Y}$, which by 
\eqref{prop:Ygood:Hor} and
\Cref{lem:Hor} implies
that $\rk(Z)=1$.
Since any complement $Z$ of $W$ in $\P_{\leq X}$ has rank one, the complements of $W$ form an antichain. This means that $W$ is a modular element of $\P_{\leq X}$.
Moreover, modularity of $W$ in $\P_{\leq X}$ implies $\rk(W) = \rk(X)+\rk(T)-\rk(Z) = \rk(\P)-1$.
In particular, $W\in\max\P(\A_Y)$ and the rank of $\P(\A_Y)$ is equal to $\rk(\P)-1$.

\imp{prop:Ygood:Hor}{prop:Ygood:Pt=Hor}
This follows immediately from \Cref{lem:subsets}.

\imp{prop:Ygood:Pt=Hor}{prop:Ygood:Pt=Pu}
Immediate.

\imp{prop:Ygood:Pt=Pu}{prop:Ygood:Bad}
Note that $\Bad\neq M(\qA)$, and so there is some $t_1\in M(\A)$ such that
$p(t_1)\notin\Bad$ and hence $\P_{t_1}=\Hor$ (by \Cref{lem:tgood}).
Then for any $t_2\in M(\A)$, we have $\P_{t_2}=\P_{t_1}=\Hor$ and hence
$p(t_2)\notin\Bad$ (by \Cref{lem:tgood}). Thus, $\Bad=\emptyset$.

\imp{prop:Ygood:Bad}{prop:Ygood:Hor}
If $\Bad=\emptyset$, then $\{X\in\P_{>T}\colon p(X)\cap
M(\qA)\neq\emptyset\}\subseteq\Hor$. 
The result follows by \Cref{lem:subsets}.

\end{description}
\end{proof}

\begin{example}
Recall the arrangement $\A$ from \Cref{ex:fib:good,ex:fib:bad}. As observed
there, the subgroup $Y=H_1$ satisfies the conditions \eqref{prop:Ygood:Pt=Hor} and \eqref{prop:Ygood:Bad} in \Cref{prop:Ygood}, while the subgroup $H_2$ does not.
This agrees with our observation in \Cref{ex:ss:mod}
that in the poset
of layers $\P=\P(\A)$, the order ideal $\P_{\leq H_1}$ is an M-ideal while
$\P_{\leq H_2}$ is not. 
\end{example}

\subsection{Fiber bundles}
We now are in the position to extend Terao's Fibration Theorem \cite{teraofibthm} from hyperplane arrangements to abelian arrangements.
\begin{theorem}[Fibration Theorem]\label{thm:fib}
The following statements are equivalent. 
\begin{enumerate}
\item\label{thm:fib:modular} 
$\P(\A_Y)$ is an M-ideal of $\P(\A)$ with $\rk(\P(\A_Y))=\rk(\P(\A))-1$.
\item\label{thm:fib:fiber} There exists an integer $\ell$ such that for any $u\in
M(\qA)$, the fiber $\pr^{-1}(u)$ is homeomorphic to $\G$ with $\ell$ points
removed.
\item\label{thm:fib:fib} $\pr:M(\A)\to M(\qA)$ is a fiber bundle.
\end{enumerate}
\end{theorem}
\begin{proof}
\
\begin{description}

\equ{thm:fib:modular}{thm:fib:fiber}
Follows from \Cref{lem:tgood} and \Cref{prop:Ygood}.

\imp{thm:fib:fiber}{thm:fib:fib} 
We need to show that the projection $\pr:M(\A)\to M(\qA)$ is locally trivial, so
let $u\in M(\qA)$.
Since $\Gamma'$ is a direct summand of $\Gamma$, we have a 
homeomorphism $T\cong(T/Y)\times Y$ giving a trivialization of the projection
$p:T\to T/Y$. 
Through the identification $p^{-1}(u)\cong \{u\}\times Y$, we 
write $\pr^{-1}(u) \cong (\{u\}\times Y)\setminus\{(u,v_1),\dots,
(u,v_\ell)\}$ for some $v_1,\dots,v_\ell\in Y$. 

By \Cref{lem:tgood}.\eqref{lem:tgood:Int}, for each $1\leq i\leq\ell$, there is a unique
$H_i\in A(\A)$ such that $(u,v_i)\in (\{u\}\times Y)\cap H_i$.
Then for each $i$, there are neighborhoods $U_i$ around $u$ in $T/Y$,
$V_i$ around $v_i$ in $Y$, and $W$ around 0 in the tangent space
$\tau_{(u,v_i)}T$ such that $(U_i\times V_i)\cap M(\A) \cong W\cap
(\tau_{(u,v_i)} T\setminus \tau_{(u,v_i)} H_i)$.
Let us pick a neighborhood $U$ around $u$ in $M(\qA)$ small enough so that
$U$ may play the role of $U_i$ above for each $i$.  
By construction, the sets $V_i$ are pairwise disjoint.

Now, for each $i$, we may define a map $\theta_i:U\times \overline{V_i}\to
\overline{V_i}$ such that $\theta_i(H_i\cap (U\times\overline{V_i})) = v_i$
and such that, for any $w\in U$, the restriction
$\theta_i|_{w\times\overline{V_i}}:w\times\overline{V_i}\to\overline{V_i}$
is a homeomorphism fixing the boundary of $\overline{V_i}$.
Extend this to a homeomorphism $\theta:U\times Y\to U\times Y$ by
$\theta(w,y)=(w,\theta_i(w,y))$ when $y\in V_i$ and $\theta(w,y)=(w,y)$ if
$y\notin\cup_i V_i$.

Finally, consider $\pr:M(\A)\to M(\qA)$.
The above map $\theta$ restricts to a homeomorphism $\pr^{-1}(U)\cong U\times
\pr^{-1}(u)$.

\imp{thm:fib:fib}{thm:fib:fiber} Immediate.
\end{description}
\end{proof}

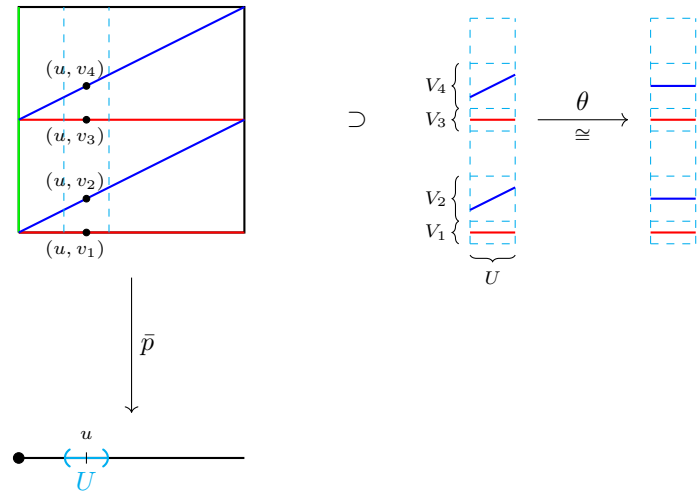
\begin{figure}[hb]
\begin{tikzpicture}[scale=3]
\draw[thick,-] (0,0)--(1,0)--(1,1)--(0,1)--(0,0);
\draw[thick,red,-] (0,0)--(1,0);
\draw[thick,red,-] (0,.5)--(1,.5);
\draw[thick,green,-] (0,0)--(0,1);
\draw[thick,blue,-] (0,0)--(1,.5);
\draw[thick,blue,-] (0,.5)--(1,1);
\draw[dashed,cyan,-] (.2,0)--(.2,1);
\draw[dashed,cyan,-] (.4,0)--(.4,1);
\node[draw,circle,fill=black,minimum size=2.5pt,inner sep=0pt] at (.3,0) {};
\node[draw,circle,fill=black,minimum size=2.5pt,inner sep=0pt] at (.3,.15) {};
\node[draw,circle,fill=black,minimum size=2.5pt,inner sep=0pt] at (.3,.5) {};
\node[draw,circle,fill=black,minimum size=2.5pt,inner sep=0pt] at (.3,.65) {};
\node at (.25,-.075) {\scriptsize $(u,v_1)$};
\node at (.25,.225) {\scriptsize $(u,v_2)$};
\node at (.25,.425) {\scriptsize $(u,v_3)$};
\node at (.25,.725) {\scriptsize $(u,v_4)$};
\draw[thick,-] (0,-1)--(1,-1);
\draw[thick,cyan,(-)] (.2,-1)--(.4,-1);
\node[cyan] at (.3,-1.1) {$U$};
\draw[-] (.3,-1.025)--(.3,-.975);
\node at (.3,-.9) {\scriptsize $u$};
\foreach \x in {2,2.2} \draw[dashed,cyan,-] (\x,-.05)--(\x,.95);
\foreach \x in {0,.5} \draw[thick,red,-] (2,\x)--(2.2,\x);
\foreach \x in {.1,.6} \draw[thick,blue,-] (2,\x)--(2.2,\x+.1);
\foreach \x in {-.05,.05,.25,.45,.55,.75,.95} 
\draw[dashed,cyan,-] (2,\x)--(2.2,\x);
\draw[decorate,decoration={brace}] (1.95,.05)--(1.95,.25);
\draw[decorate,decoration={brace}] (1.95,.45)--(1.95,.55);
\draw[decorate,decoration={brace}] (1.95,.55)--(1.95,.75);
\draw[decorate,decoration={brace}] (1.95,-.05)--(1.95,.05);
\node at (1.85,0) {\scriptsize $V_1$};
\node at (1.85,.15) {\scriptsize $V_2$};
\node at (1.85,.5) {\scriptsize $V_3$};
\node at (1.85,.65) {\scriptsize $V_4$};
\draw[decorate,decoration={brace}] (2.2,-.1)--(2,-.1);
\node at (2.1,-.2) {\scriptsize $U$};
\foreach \x in {2.8,3} \draw[dashed,cyan,-] (\x,-.05)--(\x,.95);
\foreach \x in {0,.5} \draw[thick,red,-] (2.8,\x)--(3,\x);
\foreach \x in {.15,.65} \draw[thick,blue,-] (2.8,\x)--(3,\x);
\foreach \x in {-.05,.05,.25,.45,.55,.75,.95} 
\draw[dashed,cyan,-] (2.8,\x)--(3,\x);
\draw[->] (2.3,.5) -- node[above] {$\theta$} node[below] {\scriptsize $\cong$}
(2.7,.5);
\node at (1.5,.5) {$\supset$};
\draw[->] (.5,-.2) -- node[right] {$\pr$} (.5,-.8);
\solidnodes
\node at (0,-1) {};
\end{tikzpicture}
\caption{A picture guide to the proof of local triviality in 
\Cref{thm:fib}, \eqref{thm:fib:fiber}$\implies$\eqref{thm:fib:fib}.}
\label{fig:loctriv}
\end{figure}

\begin{corollary}\label{cor:cover}
For any $\alpha\in\A\setminus\A_Y$, the projection $p:T\to T/Y$ restricts to a covering map $p|_{H_{\alpha}}:H_\alpha\to 
T/Y.$
\end{corollary}
\begin{proof}
Pick $u\in T/Y$,
and choose a neighborhood $U$ of $u$ as in the proof of \Cref{thm:fib}.\eqref{thm:fib:fiber}$\implies$\eqref{thm:fib:fib}, and consider the map $\theta$ defined there (see also \Cref{fig:loctriv}). Suppose we numbered the preimages of $u$ so that $(u,v_1),\ldots,(u,v_m)$ denote the elements of $p^{-1}(u)\cap H_\alpha$. Then for each $i=1,\ldots,m$ the homeomorphism $\theta_i^{-1}$ maps $U\times \{v_i\}$ homeomorphically to the component of $p^{-1}(U)\cap H_\alpha$ containing $(u,v_i)$. This proves the claim. 
\end{proof}

\begin{corollary}\label{cor:Yislayer}
If $Y$ is an admissible subgroup of $T$ for which $\P(\A_Y)$ is an M-ideal of $\P(\A)$ and $\dim(Y)=\dim(\G)$, 
then $Y$ is a layer of $\A$.
\end{corollary}
\begin{proof}
Because each $H_\alpha$ contains the identity of $T$, there is some $Y_0\in\tr{Y}$ which contains $Y$.
As a consequence of \Cref{thm:fib}.\eqref{thm:fib:modular},
$\rk(Y_0)=\rk(\P(\A))-1$. 
Now \[\dim(Y_0)=\dim(\G)(\rk(\P(\A))-\rk(Y_0)) = \dim(\G) = \dim(Y),\]
so we must have $Y=Y_0\in\P(\A)$.
\end{proof}

\begin{remark}
In our fibration theorem, like Terao's \cite{teraofibthm}, we require $\P(\A_Y)$ to be a corank-one subposet of $\P(\A)$. 
We conjecture that an M-ideal of any rank in $\P(\A)$ will give rise to a fiber bundle, with fiber homeomorphic to the complement of an \emph{affine} abelian arrangement, as is the case for hyperplane arrangements \cite{Paris}.
\end{remark}

\subsection{Fiber-type arrangements}

\begin{definition}\label{def:ft}
We inductively define an arrangement $\A\subseteq\Gamma$ to be \defh{fiber-type} if
$\rk\Gamma=1$ or there exists a rank-one split direct summand $\Gamma'\subseteq\Gamma$ and
$\B\subseteq\Gamma/\Gamma'$ such that $\B$ is fiber-type and the projection
$p:\Hom(\Gamma,\G)\to\Hom(\Gamma/\Gamma',\G)$ restricts to a fibration $\pr:M(\A)\to M(\B)$
whose fibers are homeomorphic to $\G$ with finitely many points removed.
\end{definition}

\begin{remark}
Following \Cref{cor:Yislayer}, if $M(\A)\to M(\B)$ is a fiber bundle,
then $\B=\qA$ for some layer $Y\in\P(\A)$.
\end{remark}

\begin{theorem}\label{thm:ft=ss}
An essential arrangement $\A$ is fiber-type if and only if 
its poset of layers $\P(\A)$ is supersolvable.
\end{theorem}
\begin{proof}
In both directions, we proceed by induction. It is clear that when $\rk\Gamma=1$,
every choice of $\A$ is both fiber-type and supersolvable. 

Now, suppose that $\A$ is fiber-type. Let $\Gamma'$ and $\B$ be as in \Cref{def:ft}, then let $T=\Hom(\Gamma,\G)$ and $Y=\ker(\pr)$. 
By induction, $\B$ is supersolvable via a chain $\zero = \Q_0\subset
\Q_1\subset\cdots\subset \Q_{n-1}=\P(\B)$ where each $\Q_i$ is an M-ideal of rank $i$ in 
$\Q_{i+1}$. Through the isomorphism $\P(\B)=\P(\qA)\cong\P(\A_Y)\subseteq\P(\A)$,
we may view each $\Q_i$ as a subposet in $\P(\A)$. Since $\P(\A_Y)$ is an M-ideal of rank $\rk(\P(\A))-1$ in $\P(\A)$ by \Cref{thm:fib}, the chain of $\Q_i$'s satisfies the conditions of \Cref{def:ss}.

Conversely, suppose that $\A$ is supersolvable via a chain $\zero = \Q_0\subset
\Q_1\subset\cdots\subset \Q_n=\P(\A)$ of M-ideals.
Since $\A$ is essential, the identity of $T$ is a maximal element of $\P(\A)$, \Cref{def:mod:2} implies that there is a unique $Y\in\max(\Q_{n-1})$ that contains the identity of $T$.
We will prove that
$\P(\A_Y)=\Q_{n-1}$; then \Cref{thm:fib} and induction will imply
the fiber-type property. Since both $\P(\A_Y)$ and $\Q_{n-1}$ are join-closed
order ideals, it suffices to prove that $A(\A_Y)=A(\Q_{n-1})$. 
This follows from the observation that both of these sets are equal to $\{H\in
A(\P)\st H\leq Y \text{ or } H\vee Y = \emptyset\}$ (cf.\ \Cref{rem:aay}).
\end{proof}

\begin{corollary}\label{cor:kpi1}
The complement of a supersolvable linear, toric, or elliptic arrangement is a $K(\pi,1)$ space.
\end{corollary}
\begin{proof}
If $\rk\Gamma=1$, then $M(\A)\cong\G\setminus\{\text{finitely many points}\}$ is
$K(\pi,1)$ when $\G=\mathbb C$, $\mathbb C^\times$ or $(S^1)^2$. Now suppose that $M(\A)\to M(\B)$ is a fiber bundle
with fiber $F\cong\G\setminus\{\text{finitely many points}\}$, a $K(\pi,1)$ as above, and $\B$ is
supersolvable. By induction, $M(\B)$ is $K(\pi,1)$. Using the homotopy long
exact sequence for the fibration, one has for $i>1$ an exact sequence
\[0=\pi_i(F)\to \pi_i(M(\A)) \to \pi_i(M(\B))=0\]
Thus, $\pi_i(M(\A))=0$ for $i>1$.
\end{proof}

\begin{example}\label{ex:fib:ft}
The arrangement $\A$ from \Cref{ex:fib:good} is fiber-type, since in
\Cref{ex:ss:ss} it was determined that its poset of layers is supersolvable.
\end{example}

\section{Geometric posets}
Up to this point, we have worked with locally geometric posets. Now we introduce a  ``global'' notion of geometricity (\Cref{def:geom}). This leads to several interesting results. First, we establish an equivalent definition of an M-ideal (\Cref{thm:niceM}) for geometric posets.
We then turn our attention to geometric semilattices. 
We show that a geometric semilattice is supersolvable if and only if its canonical extension to a geometric lattice is supersolvable (\Cref{thm:sscone}), and this has applications to affine hyperplane arrangements (\Cref{thm:affinefib}).
We then prove that supersolvability is preserved under the quotient by a group action (\Cref{thm:quotient}), which has applications to covers of abelian arrangements.

\subsection{Geometric posets and their supersolvability}

The following definition extends the notion of \defh{geometric} beyond semilattices. Indeed, a geometric semilattice in the sense of \cite{WW} is precisely a semilattice satisfying condition \ref{A}, 
as can be seen easily from comparing \ref{A} with \cite[(G4)]{WW}.

\begin{definition}\label{def:geom}
A locally geometric poset $\P$ is \emph{geometric} if for all $x,y\in\P$:
\begin{itemize}
\item[\mylabel{A}{$(\ddagger\ddagger)$}] if $\rk(x)<\rk(y)$ and $I\subseteq A(\P)$ is such that $\bigvee I \ni y$ and $\vert I \vert = \rk(y)$, then there is $a\in I$ such that $a\not \leq x$ and $a\vee x\neq\emptyset$.
\end{itemize}
\end{definition}

\begin{example} As an illustration of condition $(\ddagger\ddagger)$ in \Cref{A} consider the case of an arrangement of lines in the plane, e.g., as in \Cref{fig:lines}. 
If two lines meet at a point $y$ and a third line, that we call $x$, misses the point $y$, the first two lines cannot both be parallel to the third. In the intersection poset of the given line arrangement we will have a rank-two element $y$ that is the join of the (independent) set $I$ given by the first two lines. The line $x$ has rank one less than $y$ and must intersect one of the two lines in $I$ (which we'd call $a$ in the statement of $(\ddagger\ddagger)$). In \Cref{fig:lines} this situation is illustrated with $I=\{a,b\}$.
\end{example}

\begin{figure}[ht]
\begin{tikzpicture}
\draw[thick,->] (-.1,-.1) -- (-.1,.9);
\draw[thick,->] (-.1,-.1) -- (1.9,-.1);
\node[anchor=north] (e1) at (.9,-.1) {$e_1$};
\node[anchor=east] (e2) at (-.1,.4) {$e_2$};
\foreach \x in {0,1,2,3,4,5,6}
    \foreach \y in {0,1,2,3}
        \node (\x\y) at (\x,\y) {};
\draw[color=red,thick] (00.west) -- (60.east);
\draw[color=red,thick] (01.west) -- (61.east);
\draw[color=red,thick] (02.west) -- (62.east);
\draw[color=red,thick] (03.west) -- (63.east);
\draw[color=blue,thick] (00.center) -- (63.center);
\draw[color=blue,thick] (20.center) -- (62.center);
\draw[color=blue,thick] (40.center) -- (61.center);
\draw[color=blue,thick] (01.center) -- (43.center);
\draw[color=blue,thick] (02.center) -- (23.center);
\draw[color=green,thick] (00.south) -- (03.north);
\draw[color=green,thick] (20.south) -- (23.north);
\draw[color=green,thick] (40.south) -- (43.north);
\draw[color=green,thick] (60.south) -- (63.north);
\draw[fill] (22) circle (.2em); 
\draw[fill] (62) circle (.2em); 
\node[anchor=south east] (y) at (22) {$y$};
\node[anchor=south] (a) at (3.2,2) {$a$};
\node[anchor=south east] (b) at (3,2.4) {$b$};
\node[anchor=south] (x) at (3,.5) {$x$};
\node[anchor=west] (j) at (6.1,2) {$x\vee a$};
\end{tikzpicture}
\caption{An arrangement of lines (it should be thought of as repeating periodically in vertical and horizontal direction, generating an infinite line arrangement).}\label{fig:lines}
\end{figure}

\begin{example}
Other posets encountered already, for instance in Figures \ref{fig:locgeom} and \ref{fig:ex:ss:mod}, are geometric posets. We shall see in \Cref{cor:geoquot} that the poset of layers for any abelian Lie group arrangement is geometric.
\end{example}

For geometric posets, supersolvability can be characterized using partitions of atoms, in a way reminiscent of \cite[Remark 2.6]{falk-terao}.

\begin{theorem}\label{thm:niceM}
Let $\P$ be a geometric poset, and let $\Q$ be a pure, join-closed, proper order ideal of $\P$. 
Then $\Q$ is an M-ideal with $\rk(\Q)=\rk(\P)-1$ if and only if \begin{itemize}
    \item[\ref{dagger}] for any two distinct
$a_1,a_2 \in A(\P)\setminus A(\Q)$ and every $x\in a_1\vee a_2$ there is $a_3\in A(\Q)$ with $x > a_3$.
\end{itemize} 
\end{theorem}
\begin{proof}
By \Cref{cor:niceM}, we need only show  \ref{dagger} implies that $\Q$ is an M-ideal.
For \Cref{def:mod:1}, let $y\in \max(\Q)$ and $a\in A(\P)\setminus A(\Q)$. Take $x\in \max(\P)$ such that $a<x$. 
Take $I\subseteq A(\Q)$ such that $a\vee \bigvee I\ni x$, and $|I|=\rk(x)-1$.
By condition \ref{A}, there exists $b\in I\cup\{a\}$ such that $b\not\leq y$ and $b\vee y\neq\emptyset$.
Since $\Q$ is join-closed and $y$ is maximal in $\Q$, we cannot have $b\in \Q$ and hence $b=a$. In particular, $a\vee y\neq\emptyset$.

For \Cref{def:mod:2}, take $x\in \max\P$ and note that $\Q':=\P_{\leq x}\cap\Q$ is a join-closed order ideal in $\P_{\leq x}$, hence it is of the form $\P_{\leq y}$ for some $y< x$. 
Now \ref{dagger} holds for $\Q'$ in the geometric lattice $\P_{\leq x}$, hence \Cref{cor:niceGL}
shows that $y$ is modular of rank $\rk(\P)-1$ in $\P_{\leq x}$. This also shows that $\rk(\Q)=\rk(\P)-1$.
\end{proof}

As an immediate consequence we obtain the desired characterization of supersolvability in geometric posets.
\begin{corollary}\label{cor:nice:GP}
Let $\P$ be a geometric poset. Then $\P$ is supersolvable if and only if there is a chain $\{\zero\}=\Q_0\subset \Q_1 \subset \ldots \subset \Q_n=\P$ of pure, join-closed order ideals of $\P$ with $\rk(\Q_i)=i$ and so that \ref{dagger} holds for $\Q_{i-1}$ in $\Q_i$, for all $i=1,\ldots,n$.
\end{corollary}

\subsection{Supersolvable geometric semilattices}
A main source of intuition on geometric semilattices comes from finite-dimensional vector spaces. The poset of intersections of any arrangement of hyperplanes in a vector space is a geometric semilattice, as is the poset of all affine subspaces of a finite-dimensional vector space. This suggests that geometric semilattices should be an ``affine'' counterpart to geometric lattices (among whose main examples we find posets of intersections of arrangements of {\em linear} hyperplanes). The following structure theorem is an abstract counterpart to this linear-affine relationship.

\begin{theorem}[Wachs-Walker \cite{WW}]\label{thm:WW}
A poset $\P$ is a geometric semilattice if and only if $\P=\L\setminus \L_{\geq a}$ where $\L$ is a geometric lattice and $a$ is an atom of $\L$. The poset $\L$ and the element $a$ are uniquely determined by $\P$, up to isomorphism.
\end{theorem}

We will use this theorem to relate M-ideals in geometric semilattices to the classical notion of modular elements in geometric lattices.

For the remainder of this section let $\P$ be a chain-finite geometric semilattice and let $\L$ be a geometric lattice with an atom $a_0\in A(\L)$ such that $\P=\L\setminus \L_{\geq a_0}$.  The meet operation coincides in $\L$ and $\P$. When needed, we will distinguish the join operations as $\vee_\P$ and $\vee_\L$.

\begin{lemma}\label{lem1trivial}
A subposet $\Q$ of a geometric semilattice $\P=\L\setminus\L_{\geq a_0}$ is a pure, join-closed order ideal satisfying \Cref{def:mod:1} in $\P$ if and only if $\Q=\P_{\leq x}:=\{p\in\P\st p\leq_{\L}x\}$ for some $x\in\L_{\geq a_0}$.
\end{lemma}
\begin{proof}
First, assume $x\in\L_{\geq a_0}$ and let $\Q=\P_{\leq x}$. Clearly $\Q$ is a pure, join-closed order ideal. To verify \Cref{def:mod:1}, let $a\in A(\P)$ and $y\in\Q$ such that $a\vee_{\P}y=\emptyset$. This means that $a\vee_{\L}y\geq a_0$, and hence $a\vee_{\L}y = a_0\vee_{\L}a\vee_{\L}y \geq a_0\vee_{\L}y>y$. Since $a\vee_{\L}y$ covers $y$, we must have $a\vee_{\L}y=a_0\vee_{\L}y$. Thus, $a\leq a\vee_{\L}y=a_0\vee_{\L}y\leq x$ and hence $a\in \Q$.

Conversely, assume that $\Q$ is a pure, join-closed order ideal of $\P$ satisfying \Cref{def:mod:1}, and let $u=:\bigvee_{\L}\Q$ and $x:=u\vee_{\L}a_0$. It is clear that $\Q\subseteq\P_{\leq x}$, but to obtain equality we consider two cases: either $u<x$ or $u=x$.

If $u<x$, then $u\in\P$ and hence $\Q=\P_{\leq u}$. We thus need to show that $\P_{\leq x}\subseteq\P_{\leq u}$. Let $y\in\P_{\leq x}$, so $y\in\P$ and $y\leq x=u\vee_{\L}a_0$. To argue by contradiction, suppose that $y\not\leq u$. 
Then there exists $a\in A(\P)$ such that $a\leq y$ and $a\not\leq u$. This implies $u\lessdot u\vee_{\L}a\leq u\vee_{\L}a_0$, where the latter inequality holds because $a\leq y\leq u\vee_{\L}a_0$. 
Since $u\lessdot u\vee_{\L}a_0$, this implies that $u\vee_{\L}a = u\vee_{\L}a_0\geq a_0$. However, $u\vee_{\L}a\geq a_0$ implies $u\vee_{\P}a=\emptyset$, contradicting the assumption that $\Q=\P_{\leq u}$ satisfies \Cref{def:mod:1}.

For the second case, suppose $u=x\geq a_0$. 
Then, e.g., by \cite[Theorem 2.29 and 6.4.(iii)]{aigner}, there exists a set $I\subseteq A(\Q)$ such that $x=a_0\vee_{\L}\bigvee_{\L}I$ and $\rk(x)=|I|+1$. But then $\bigvee_{\L}I\not\geq a_0$, and hence $\bigvee_{\P}I\neq\emptyset$. Since $\Q$ is join-closed, this means $y:=\bigvee_{\P}I\in\Q$. Moreover, since $y\vee_{\L}a_0=x$, we have $y\lessdot x$.
Now suppose that $a\in A(\P_{\leq x})\setminus A(\Q)$. Since $\Q$ satisfies \Cref{def:mod:1}, we must have $a\vee_{\L}y\not\geq a_0$, so the second inequality in the following chain is strict $x\geq a_0\vee_{\L}a\vee_{\L}y>a\vee_{\L}y>y$, contradicting the fact that $y\lessdot x$. This means that no such $a$ can exist, i.e. $A(\P_{\leq x})\subseteq A(\Q)$. Since $\Q$ is join-closed and every element of $\P_{\leq x}$ is a join of atoms, we must then have $\P_{\leq x}\subseteq\Q$.
\end{proof}

\begin{lemma}\label{lemcompL} 
Let $\P=\L\setminus\L_{\geq a_0}$ be a geometric semilattice, $x\in\L_{\geq a_0}$, and $\Q=\P_{\leq x}$.
If $z$ is a complement to $x$ in $\L$, then $z\in\P$ and $z$ is a complement to $\max\Q$ in $\P$.
Moreover, every complement to $\max\Q$ in $\P$ is a complement to $x$ in $\L$.
\end{lemma}
\begin{proof}
 Let $z$ be a complement to $x$ in $\L$. From $x\wedge_\L z=\zero$ we have $z\in \P$. Now consider any $m\in \max \Q$. In particular,  $m\vee_\L a_0=x$.  Then 
 $m\vee_\L z \lessdot
 (a_0 \vee_\L m\vee_\L z) =
 x\vee_\L z
 =\one
 $, where the first inequality holds because $\L$ is a geometric lattice (see \Cref{def:geomlattice}). In particular $m\vee_\P z\subseteq \max \P$. On the other hand, we have $x \wedge_\L z \geq m \wedge_\L z= m\wedge_\P z$,  thus $x\wedge_\L z=\zero$ implies $m\wedge_\P z=\zero$. 
 
Let $z\in \P$ be a complement to $\max\Q$ in $\P$. For every $q\in \max\Q$ we have $q\vee_\L a_0=x$, and $q\vee_\L z\geq p$ for some $p\in \max\P$. Thus, $x\vee_\L z = (a_0\vee_\L q \vee_\L z) \geq a_0\vee_\L p = \one$. On the other hand, surely $x\wedge z \subseteq \P_{\leq x}=\Q$ and thus there is $q\in \max \Q$ such that  $x\wedge z\leq q $. But then,  $x\wedge z = x\wedge z \wedge z \leq q\wedge z =\zero$. This proves that $z$ is a complement to $x$ in $\L$.
\end{proof}

\begin{theorem}\label{thm:sscone}
Let $\P=\L\setminus\L_{\geq a_0}$ be a geometric semilattice. A subposet $\Q$ of $\P$ is an M-ideal if and only if $\Q=\P_{\leq x}$ for some modular element $x$ in $\L$ with $x\geq a_0$.
Consequently, $\P$ is supersolvable if and only if $\L$ is supersolvable via a chain of modular elements passing through $a_0$.
\end{theorem}

\begin{proof} 
By \Cref{lem1trivial}, it suffices to prove that for $x\in\L_{\geq a_0}$, the order ideal $\Q=\P_{\leq x}$ satisfies \Cref{def:mod:2} if and only if $x$ is modular in $\L$.

First, suppose that $\Q$ is an M-ideal, and we show that $x$ is modular in $\L$.
If $x$ is not modular in $\L$
there are two complements $z',z''$ of $x$ in $\L$ such that $z'\leq z''$. By
\Cref{lemcompL}, $z',z''\in \P$ and both are complements in $\P$ to every $y\in \max\Q$. Since $\Q$ is an M-ideal we can choose $m\in \max\P_{\geq z''}$ and $y\in \max\Q_{\leq m}$. Then $y$ is modular in $\P_{\leq m}$ and $z',z''$ are comparable complements of $y$ in $\P_{\leq m}$ -- a contradiction. Thus, $x$ is a modular element in $\L$.

Now assume that $x$ is modular in $\L$.  
Take $m\in \max \P$ and consider $y:= x\wedge_\L m$. Since $m\lessdot \one$,  modularity of $x$ implies $\rk(x)-\rk(y)=1$ and hence $y\lessdot x$. Thus $y\in \max \Q$, and $x\wedge z = y\wedge z$ for all $z\in \P_{\leq m}$ 
(indeed: by modularity $(x\wedge z)\vee y = x \wedge (z\vee y) \geq y$, and  since $z\vee y\not \geq x$ since $z\vee y\in \P_{m}$, $y\lessdot x$ implies $(x\wedge z)\vee y = y$, hence $(x\wedge z)\leq y$ and the claim follows). In particular, for $z\in \P_{\leq m}$ we have that $z\wedge y=\zero$ implies $z\wedge x=0$. Since trivially $z\vee y=\one$ implies $z\vee x=\one$, we have that every complement of $y$ in $\P_{\leq m}$ is a complement of $x$ in $\L$, thus modularity of $x$ implies modularity of $y$.

The claim about supersolvability follows because the chain $\zero\subset \P_{\leq x_1}\subset\cdots\subset\P_{\leq x_n}=\P$ will satisfy \Cref{def:ss} in $\P$ if and only if the chain $\zero\subset\L_{\leq a_0}\subset\L_{\leq x_1}\subset\cdots\subset\L_{\leq x_n}=\L$ does in $\L$.
\end{proof}

\subsection{Affine hyperplane arrangements}
Let $V\cong\C^n$ be a complex vector space, and let $\A$ be an arrangement of affine hyperplanes in $V$. 
Associated to $\A$ is a polynomial $f_\A\in\C[x_1,\dots,x_n]$ whose solution set is the union of the hyperplanes in $\A$.
Denote by $f_{c\A}\in\C[x_0,x_1,\dots,x_n]$ the homogenization of $f_\A$. 
The \emph{cone} of $\A$, denoted by $c\A$, is the linear hyperplane arrangement in $\C^{n+1}$ whose hyperplanes are the (linear) components of the solution set of $f_{c\A}$.
The cone $c\A$ has one more hyperplane than $\A$ does, namely $H_0=\ker(x_0)$.
This coning construction can be reversed to define the \emph{de-cone} $d\A$ of any linear hyperplane arrangement, see \cite{OT}.

\begin{remark}
The poset of layers of the affine arrangement $\A$ has the structure of a geometric semilattice since $\P(\A)\cong \P(c\A)\setminus\P(c\A)_{\geq H_0}$. Note that \Cref{cor:nice:GP} shows that $\P(\A)$ is supersolvable according to \Cref{def:ss} exactly when $\A$ is supersolvable in the sense of the definition given by Falk and Terao  \cite[Remark 2.6]{falk-terao}.
\end{remark}

The following extends \Cref{def:ft} to affine arrangements:
\begin{definition}
We inductively define an affine hyperplane arrangement $\A$ in a complex vector space $V$ to be \defhno{fiber-type} if either $\dim(V)=1$ or there is a choice of coordinates $V\cong\C^n$ and an arrangement $\B$ in $\C^{n-1}$ such that the projection $p:\C^n\to \C^{n-1}$ onto the first $(n-1)$ coordinates restricts to a fiber bundle $M(\A)\to M(\B)$ whose fibers are homeomorphic to $\C$ with finitely many points removed.
\end{definition}

This definition allows us to be obtain an affine analogue to Terao's Fibration Theorem \cite{teraofibthm}.

\begin{theorem}\label{thm:affinefib}
Let $\A$ be an essential arrangement of affine hyperplanes in a complex vector space. Then $\P(\A)$ is supersolvable if and only if $\A$ is fiber-type.
\end{theorem}
\begin{proof}
We proceed by induction on the rank of $\A$.
Notice that every rank-one locally geometric poset is supersolvable, and every rank-one affine arrangement is fiber-type. 
Let $H_0$ be the additional hyperplane in $c\A$ so that $\P(\A)=\P(c\A)\setminus\P(c\A)_{\geq H_0}$, which under our choice of coordinates on $\C^{n+1}$ will be the first coordinate hyperplane. 

Assume that $\P(\A)$ is supersolvable, which implies by \Cref{thm:sscone} that there is a modular chain $\zero<H_0=Y_1<Y_2<\cdots<Y_{n-1}<Y_n=\one$ in $\P(c\A)$. Abbreviate $Y:=Y_{n-1}$.%

Then we may pick coordinates on the ambient vector space of $c\A$ so that the projection $p:\C^{n+1}\to\C^n$ is the quotient by $Y$,
and so that this restricts to a fiber bundle $M(c\A)\to M(c\qA)$.
The following composition is then a fiber bundle
\begin{center}
\begin{tikzpicture}[scale=0.825, every node/.style={scale=0.8}]
\node (2) at (0,0) {$M(\A)\times\C^\times$};
\node (3) at (4,0) {$M(c\A)$};
\node (4) at (8,0) {$M(c\qA)$};
\node (5) at (12,0) {$M(d(c\qA))\times\C^\times$};
\draw[->] (2) edge node[above]{$\cong$} (3);
\draw[->] (3) edge node[above]{$p$} (4);
\draw[->] (4) edge node[above]{$\cong$} (5);
\node (a) at (0,-.75) {$((x_1,\dots,x_n),x_0)$};
\node (b) at (4,-.75)  {$(x_0,x_0x_1,\dots,x_0x_n)$};
\node (c) at (8,-.75) {$(x_0,x_0x_1,\dots,x_0x_{n-1})$};
\node (d) at (12,-.75) {$((x_1,\dots,x_{n-1}),x_0)$};
\draw[|->] (a)--(b);
\draw[|->] (b)--(c);
\draw[|->] (c)--(d);
\end{tikzpicture}
\end{center}
where the first and last maps are the standard cone/de-cone homeomorphisms (see eg. \cite[Proposition 5.1]{OT}).
The subbundle obtained by setting $x_0=1$ is then a fiber bundle $M(\A)\to M(d(c\qA))$.
Now $\P(c\qA)$ is isomorphic to the subposet $\P(c\A)_{\leq Y}$, therefore $c\qA$ is supersolvable via the chain of modular elements $\zero<H_0/Y<Y_2/Y<\cdots<Y_{n-1}/Y$, and its de-cone $d(c\qA)$ with respect to $H_0/Y$ is supersolvable by \Cref{thm:sscone}. 
By induction, $d(c\qA)$ is fiber-type. Therefore, $\A$ will also be fiber-type.

Conversely, assume that $\A$ is fiber-type, and choose coordinates so that the projection $\C^n\to \C^{n-1}$ restricts to a fiber bundle $p:M(\A)\to M(\B)$ for some fiber-type arrangement $\B$ in $\C^{n-1}$.
The composition
\begin{center}
\begin{tikzpicture}[scale=0.825, every node/.style={scale=0.8}]
\node (1) at (0,0) {$M(c\A)$};
\node (2) at (4,0) {$M(\A)\times\C^\times$};
\node (3) at (8,0) {$M(\B)\times\C^\times$};
\node (4) at (12,0) {$M(c\B)$};
\draw[->] (1) edge node[above]{$\cong$} (2);
\draw[->] (2) edge node[above]{$p\times\operatorname{id}$} (3);
\draw[->] (3) edge node[above]{$\cong$} (4);
\node (a) at (0,-.75) {$(x_0,x_1,\dots,x_n)$}; 
\node (b) at (4,-.75)  {$\left(\left(\frac{x_1}{x_0},\dots,\frac{x_n}{x_0}\right),x_0\right)$};
\node (c) at (8,-.75) {$\left(\left(\frac{x_1}{x_0},\dots,\frac{x_{n-1}}{x_0}\right),x_0\right)$};
\node (d) at (12,-.75) {$(x_0,x_1,\dots,x_{n-1})$};
\draw[|->] (a)--(b);
\draw[|->] (b)--(c);
\draw[|->] (c)--(d);
\end{tikzpicture}
\end{center}
is then a fiber bundle, which we denote by $\hat{p}$. 
This fiber bundle is the quotient by some dimension-one modular element $Y\in \P(c\A)$, which is necessarily contained in the additional hyperplane $H_0$.
Since $\B$ is fiber-type, $\P(\B)$ is supersolvable by induction. 
Then by \Cref{thm:sscone}, $\P(c\B)$ is supersolvable via a chain of modular elements passing through the additional hyperplane $\hat{p}(H_0)$ of $c\B$.
Via the poset isomorphism $\P(c\A)_{\leq Y}\cong \P(c\B)$ induced by $\hat{p}$, and using modularity of $Y$ in $\P(c\A)$, \cite[Proposition 2.2.1a)]{ziegler} yields a chain of modular elements in $\P(c\A)$ passing through $H_0$.%

Again by \Cref{thm:sscone}, we conclude that $\P(\A)$ is supersolvable.
\end{proof}

\subsection{Group quotients and topological covers}

In the following we study the behaviour of supersolvability of posets with respect to certain types of group actions.

Let $G$ be a group. An \textit{action} of $G$ on a poset $\P$ is any group homomorphism $G\to\Aut(\P)$ from $G$ to the group of automorphisms of $\P$. Given a group element $g\in G$ it is customary to denote the associated automorphism by $g:\P\to\P$. For $x\in \P$ we will then often write $gx$ for $g(x)$. Following \cite{dD}, we will focus on the following special type of action.

\begin{definition}
\label{def:translative}
Let $\P$ be a poset with an action of a group $G$. We call the action \defh{translative} if $x\vee gx\neq\emptyset$ implies $x=gx$ for every $x\in \P$ and every $g\in G$.
\end{definition}

Translative actions were introduced in \cite{DR} in order to model periodic hyperplane arrangements, as we illustrate in the next example.

\begin{example}\label{ex-translation} Consider the periodic arrangement of lines represented in \Cref{fig:lines}. The standard generators $e_1,e_2$ of the group $G=\mathbb Z^2$ act via the traslations given by the arrow in the picture. This action is translative. Indeed, for any $g\in G$ and any line $x$, the lines $gx$ and $x$ have the same direction. Hence, if they intersect (i.e., if $x\vee gx \neq \emptyset$), they must be identical ($x=gx$). 
\end{example}

Write $Gx=\{gx\st g\in G\}$ for the orbit of an $x\in\P$ under $G$. If a group $G$ acts on a poset $\P$ we define the set of orbits $\P/G:=\{Gx\st x\in \P\}$. On it we consider the relation given by $Gx\leq Gy$ if there is $g\in G$ with $x\leq gy$.

\begin{lemma}[{\cite[Lemmas 2.12 and 2.13]{dD}}] \label{lociso}
Let $\P$ be a poset with a translative action of a group $G$. Then the relation $\leq$ on $P/G$ is a partial order relation. Moreover, for every $z\in \P$ the function $f_z:\P_{\leq z}\to (\P/G)_{\leq Gz}$, $p\mapsto Gp$, defines an isomorphism of posets. 
\end{lemma}

\begin{example}[{\cite[Section 9]{dD}}]\label{ex:UC}
Let $\A$ be an abelian Lie group arrangement of rank $n$ in $\G=(S^1)^d\times \R^v$. Then the lift of all $H_\alpha$, $\alpha\in\A$ to the universal cover of $T$ is an arrangement $\A^\upharpoonright$ of affine subspaces in $\mathbb R^{(d+v)n}$. Its poset of layers $\P(\A^\upharpoonright)$ is a geometric semilattice and the action on $\A^\upharpoonright$ of the group of deck transformations induces a translative action of $\mathbb Z^{nd}$ on $\P(\A^\upharpoonright)$. Then, $\P(\A)$ is isomorphic to the quotient  $\P(\A^\upharpoonright)/\mathbb Z^{dn}$. For instance, \Cref{fig:lines} depicts the arrangement $\A^\upharpoonright$ for the toric arrangement $\A$ of \Cref{fig:loctriv}.
\end{example}

\begin{example}\label{ex:IC}
Let $\Gamma_1$ be a finitely-generated free abelian group, and let $\Gamma_2$ be a subgroup of $\Gamma_1$ of finite index. Let $\G$ be a connected abelian Lie group and let $\A\subseteq \Gamma_2$ be a finite subset. 
Call $\P_1$, $\P_2$ the posets of layers of the arrangements defined by $\A$ in  $T_1:=\Hom(\Gamma_1,\G)$, resp.\ $T_2:=\Hom(\Gamma_2,\G)$. 
Now since $\G$ is a product of copies of the injective $\mathbb Z$-modules $S^1$ and $\R$, $\G$ is injective itself and so $\Hom(-,\G)$ is exact. This implies that the inclusion $\Gamma_2\hookrightarrow \Gamma_1$ induces a covering map $T_1\to T_2$ whose group of deck transformations is the discrete subgroup $G$ of $T_1$ that is the image of the inclusion  
$\Hom(\Gamma_1/\Gamma_2,\G)\hookrightarrow\Hom(\Gamma_1,\G)$.
Recall that $\A\subseteq\Gamma_2$, thus for every $\alpha\in\A$ and $g\in G$ one has $g(\alpha)=0$.
This implies $GH_\alpha=H_\alpha$ for every $\alpha\in\A$, and so $G$ acts on $\P_1$.
Moreover, since every layer $Y$ of $\A$ in $T_1$ is a coset of a subgroup, its image under a deck transformation is another coset of the same subgroup -- hence either identical or disjoint with $Y$. This shows that the induced action of the deck transformations is translative on $\P_1$, and $\P_2$ is the quotient of $\P_1$ by this action. 
\end{example}

\begin{lemma}\label{lem:geometriquot} Let $\P$ be a locally geometric poset with a translative action of a group $G$. Then, $\P/G$ is a locally geometric poset. Moreover, if $\P$ is geometric then so is $\P/G$.
\end{lemma}
\begin{proof}
The first claim is an immediate corollary of \Cref{lociso}. For the second claim, suppose $\P$ satisfies \ref{A} and let $f:\P\to\P/G$ denote the quotient map. Let $x,y\in \P/G$ with $\rk(x)<\rk(y)$ and choose $I\subseteq A(\P/G)$ with $\rk(y)=\vert I \vert$ and $y\in \bigvee I$. Choose $y'\in f^{-1}(y)$ and $x'\in f^{-1}(x)$. Since the $G$-action preserves rank, $\rk(x)=\rk(x')$ and $\rk(y)=\rk(y')$. Let $I'$ be the preimage of $I$ under the local isomorphism $\P_{\leq y'}\simeq (\P/G)_{\leq y}$ given by \Cref{lociso}. Then $\vert I'\vert=\rk(y')>\rk(x')$. Using property \ref{A} in $\P$ we can choose $a'\in I'$ such that $a'\not\leq z$ and $a'\vee x'\neq\emptyset$. Choose $z'\in a'\vee x'$ and let $a:=f(a')\in I$. Now $f(z')\geq a$ and $f(z')\geq f(x')=x$. Moreover, $a\in I$ and, with \Cref{lociso} applied to $\P_{\leq z'}$, $a\not\leq x$. Thus $\P/G$ satisfies \ref{A}.
\end{proof}

\begin{corollary}\label{cor:geoquot}
Let $\A$ be any abelian Lie group arrangement. Then the poset $\P(\A)$ is geometric.
\end{corollary}
\begin{proof}
The claim follows from \Cref{lem:geometriquot} because $\P(\A)$ is a quotient  of a geometric semilattice (see \cite[Lemma 9.2.(ii)]{dD}) by a translative action, and geometric semilattices are geometric posets.
\end{proof}

\begin{lemma} \label{lemvee}
Let $\P$ be a poset with a translative action of a group $G$. Then, for all $x_1,x_2\in \P$,
$$Gx_1\vee_{\P/G} Gx_2 = \left.\left(\bigcup_{g_1,g_2\in G} g_1x_1\vee_\P g_2x_2\right)\right/G.$$ 
\end{lemma}
\begin{proof}  
Let $z\in g_1x_1\vee_\P g_2x_2$ for some $g_1,g_2\in \P$.  
In particular, $Gz\geq Gx_i$ for $i=1,2$ and, for every $Gp\in\P/G $ with $Gz\geq Gp\geq Gx_i$, $i=1,2$, it must be $z \geq_\P f_z^{-1}(p) \geq_\P g_ix_i$, $i=1,2$ (where $f_z$ is the isomorphism of \Cref{lociso}). Then, by definition of $z$, $f_z^{-1}(p)=z$. This means $Gp=Gz$ and proves the right-to-left inclusion in the claim.

For the left-to-right inclusion, let $Gz\in Gx\vee_{\P/G} Gy$. Then, by definition, there are $g_1,g_2\in G$ with $z\geq g_ix_i$ for $i=1,2$. Now for every $p\in\P$ with $z\geq p\geq g_ix_i$, $i=1,2$, we must have $Gx_i\leq f_z(p)\leq f_z(z)=Gz$ (again $f_z$ is the isomorphism of \Cref{lociso}). Now $Gz\in Gx_1\vee Gx_2$ implies $f_z(p)=f_z(z)$, and bijectivity of $f_z$ shows $p=z$, whence $z\in g_1x_1\vee_{\P} g_2x_2$.
\end{proof}

\begin{lemma} Let $\P$ be a chain-finite poset with an action of a group $G$ and let $\Q$ be a $G$-invariant subposet (i.e., $G\Q=\Q$). 
\begin{itemize}
\item[(i)] $\Q$ is an order ideal if and only if $\Q/G$ is.
\item[(ii)] $\Q$ is pure if and only if $\Q/G$ is.
\item[(iii)] If the action is translative, $\Q$ is join-closed if and only if $\Q/G$ is.
\end{itemize}
\end{lemma}
\begin{proof} 
Since $G$ acts by order-preserving automorphisms, (i) is immediate and (ii) follows since those automorphisms preserve chain length. Claim (iii) is a consequence of \Cref{lemvee}.
\end{proof}

\begin{lemma} \label{lem:quotient1}
Let $\P$ be a locally geometric poset with a translative action of
a group $G$ and let $\Q$ be a $G$-invariant subposet of $\P$.  If $\Q\subseteq
\P$ satisfies \Cref{def:mod:1}, then so does $\Q/G\subseteq \P/G$. When $\P$
satisfies \ref{A}, the converse also holds.
\end{lemma}
\begin{proof} 
Let \Cref{def:mod:1} hold for $\Q$, and consider $Gy\in \Q/G$, $Ga\in A(\P/G)$ such that $Gy\vee_{\P/G} Ga = \emptyset$. Then $y\in Q$, $a\in A(\P)$ and, by \Cref{lemvee}, $y\vee_\P a=\emptyset$. By assumption then $a\in \Q$ and so $Ga\in \Q/G$. 

For the reverse implication, suppose that $\P$ satisfies \ref{A} and let
\Cref{def:mod:1} hold for $\Q/G$. Pick $y\in Q$, $a\in A(\P)$ such that
$y\vee_\P a=\emptyset$. Now if $Gp\in Gy\vee Ga$ then we can choose $p$ and
$g\in G$ such that $p\in y\vee ga$. 
Since the action preserves rank, $\rk(gy)<\rk(p)$ and $\rk(ga)=1$. Because $\P$ is locally geometric, then, $y\lessdot p$.
Now
translativity and \ref{A} imply that $gy\vee_\P ga\neq\emptyset$ contradicting the choice of $y$ and $a$. Therefore $Gy\vee Ga=\emptyset$ and by assumption $Ga\in \Q/G$, which means $a\in \Q$.
\end{proof}

\begin{lemma}\label{lem:quotient2}
Let $\P$ be a locally geometric poset with a translative action of a group $G$ and let $\Q$ be a $G$-invariant subposet of $\P$. Then, $\Q\subseteq \P$ satisfies \Cref{def:mod:2} if and only if $\Q/G\subseteq \P/G$ does.
\end{lemma}
\begin{proof} For every $Gx\in \max \P/G$, \Cref{lociso} shows that the quotient map defines a poset isomorphism between $\P_{\leq x}$ and $(\P/G)_{\leq Gx}$. In particular, given $y\in \P_{\leq x}$, every complement of some $Gy\in (\P/G)_{\leq Gx}$ is the orbit of a complement of $y$, and viceversa. Therefore $y$ is modular in $\P_{\leq x}$ if and only if $Gy$ is modular in $(\P/G)_{\leq Gx}$. Moreover, obviously $y\in \Q$ if and only if $Gy\in \Q/G$. Noting that the orbit of every $x\in \max\P$ is maximal in $\P/G$ proves the claim.
\end{proof}

\begin{theorem}\label{thm:quotient}
	Let $\P$ be a locally geometric poset with a translative action of a
group $G$ and let $\Q$ be a $G$-invariant subposet of $\P$. If $\Q$ is an
M-ideal in $\P$, then $\Q/G$ is an M-ideal in $\P/G$. Moreover, if $\P$
satisfies \ref{A}, the converse also holds.
\end{theorem}
\begin{proof}
This is a combination of \Cref{lem:quotient1,lem:quotient2}.
\end{proof}

\begin{corollary}\label{cor:coverdown}
Let $\Gamma_1$ be a finitely-generated free abelian group, and let $\Gamma_2$ be a subgroup of $\Gamma_1$ of finite index. Let $\G$ be an abelian Lie group and let $\A\subseteq \Gamma_2$ be a finite subset. Call $\P_1$, $\P_2$ the posets of layers of the arrangements defined by $\A$ in  $\Hom(\Gamma_1,\G)$, resp.\ $\Hom(\Gamma_2,\G)$. 
Then $\P_1$ is supersolvable if and only if $\P_2$ is.
\end{corollary}
\begin{proof} Both posets are geometric by \Cref{cor:geoquot}.
From \Cref{ex:IC} we know that $\P_2$ is the quotient of $\P_1$ by the  translative action of the group $G$ of deck transformations of the covering $\Hom(\Gamma_1,\G)\to \Hom(\Gamma_2,\G)$, and that $GH_\alpha=H_\alpha$ for every $\alpha\in A$. The last item implies that $\tr{Y}$ is $G$-invariant 
for every layer $Y\in \P_1$. In particular, if $\P_1$ is supersolvable, all members of the associated chain of M-ideals are $G$-invariant. On the other hand, if $\P_2$ is supersolvable, the preimage of every element of the associated chain of M-ideals is obviously $G$-invariant. Now the claimed equivalence follows by \Cref{thm:quotient}.
\end{proof}

\begin{corollary}\label{cor:coverup}
Recall the setup of \Cref{ex:UC}. For every abelian Lie group arrangement $\A$, the poset $\P(\A^\upharpoonright)$ is supersolvable if and only if $\P(\A)$ is.
\end{corollary}
\begin{proof}
\def\Hup{\widetilde{H}}
The proof is similar to that of \Cref{cor:coverdown}. We already know that both posets are geometric. For every $\alpha\in \A$ 
the group $G= \mathbb Z^{dn}$ of deck transformations acts on the union of all lifts of $H_\alpha$ to the universal cover of $T$.
Hence $\tr{Y}$ is $G$-invariant for all layers $Y\in\P(\A^\upharpoonright)$. Thus, all M-ideals of $\P(\A^\upharpoonright)$ are $G$-invariant. On the other hand, obviously the lift of every M-ideal of $\P(\A)$ is $G$-invariant. The claimed equivalence follows by \Cref{thm:quotient}.
\end{proof}

\begin{remark} \Cref{ex:UC} describes the lift of an abelian arrangement to the universal cover of its ambient space. \Cref{cor:coverup} can be generalized to other covering spaces in the following way. Let $\G_1$, $\G_2$ be two connected abelian Lie groups and suppose that $\G_2$ is a topological cover of $\G_1$. Then $\G_1$ is isomorphic to the quotient $\G_2/L$, where $L$ is a discrete subgroup of $\G_2$. Let $\Gamma$ be a finitely generated free Abelian group. 
Then $\Hom(\Gamma,-)$ is exact, hence $\Hom(\Gamma,\G_1)$ is the quotient of the topological group $\Hom(\Gamma,\G_2)$ by its discrete subgroup $\Hom(\Gamma,L)$. 
In particular, $\Hom(\Gamma,\G_2)$ is a topological covering of $\Hom(\Gamma,\G_1)$.
 
Let $\A\subseteq \Gamma$ be a finite subset. Call $\P_1$ the poset of layers of the arrangement defined by $\A$ in $\Hom(\Gamma,\G_1)$, and let $\P_2$ be the poset of layers of the lift of that arrangement to the cover  $\Hom(\Gamma,\G_2)$ of $\Hom(\Gamma,\G_1)$. Then, $\P_1$ is supersolvable if and only if $\P_2$ is. This follows from \Cref{cor:coverup} because $\G_1$ and $\G_2$ have the same universal cover, thus both arrangements lift to the same $\A^\upharpoonright$.
\end{remark}

\section{Strict supersolvability and configuration spaces}

\subsection{Definitions}
Several phenomena that are well-known about supersolvable geometric lattices in
fact do not hold in our general setting. The difficulty lies in \Cref{def:mod:1},
which may be rephrased as requiring  $|a\vee y|\geq 1$ for any $y\in\Q$ and $a\in
A(\P)\setminus A(\Q)$. If instead we require such $a$ and $y$ to have a
\textit{unique} minimal upper bound, these phenomena do indeed appear.
In other words, we need a stronger notion than an M-ideal, which we define next.

\begin{definition}\label{def:TM}
A \defh{TM-ideal} of a locally geometric poset $\P$ is a pure,
join-closed, order ideal $\Q \subseteq\P$ such that:
\begin{deflist}
\item\label[definition]{def:TM:1}
$|a\vee y|=1$ whenever $y\in \Q $ and $a\in A(\P)\setminus A(\Q )$. 
\item\label[definition]{def:TM:2}
for every $x\in\max(\P)$, there is some $y\in\max(\Q )$ such that $y$ is a modular
element in the geometric lattice $\P_{\leq x}$.
\end{deflist}
\end{definition}

\begin{example}
$\{\zero\}$ and $\P$ are always TM-ideals of $\P$.
\end{example}

\begin{remark}\label{rem:TM}
When $\P$ is the poset of layers for an abelian arrangement $\A$ in $T$, \Cref{def:TM} has the following topological interpretation.
Let $Y$ be an admissible subgroup of $T$ for which $\P(\A_Y)$ is an M-ideal of $\P(\A)$. The following statements are equivalent.
\begin{enumerate}
\item \label{thm:monodromy:TM}
$\P(\A_Y)$ is a TM-ideal of $\P(\A)$ with $\rk(\P(\A_Y))=\rk(\P(\A))-1$.
\item \label{thm:monodromy:int}
For every $Y'\in\tr{Y}$ and $H\in A(\A\setminus\A_Y)$, the intersection $Y'\cap H$ is
connected.
\end{enumerate}
\end{remark}

\begin{definition}\label{def:sss}
A locally geometric poset $\P$ is \defh{strictly supersolvable} if there is a
chain \[
\zero = \Q_0 \subset \Q_1 \subset \cdots \subset \Q_n = \P\]
 where each $\Q_i$ is a TM-ideal of $\Q_{i+1}$ with $\rk(\Q_i)=i$.
\end{definition}

\begin{example}
Any rank-one locally geometric poset is strictly supersolvable.
\end{example}

\begin{example}\label{ex:MnotTM}
The poset $\P$ from \Cref{ex:ss:mod} (see \Cref{fig:ex:ss:mod}) is not strictly
supersolvable. Indeed, its only proper M-ideals are $\P_{\leq 1}$ and $\P_{\leq
3}$, and the fact that $|1\vee3|=2$ means neither is a TM-ideal.
\end{example}

\begin{example}\label{ex:sss}
Let $\Gamma=\mathbb Z^2$ and $\A=\{\alpha_1=(1,0),\alpha_2=(0,2),\alpha_3=(1,2)\}$. Let $\G=S^1\times \mathbb R^v$ and
consider an arrangement in $T$.
If $v=0$ or $v=1$, we may identify
$T$ with $(S^1)^2$ or $(\C^\times)^2$, respectively.
\Cref{fig:loctriv} depicts the arrangement in $(S^1)^2$ and \Cref{fig:ex:sss} depicts the Hasse diagram for
its poset of layers.
This poset has a proper TM-ideal, whose maximal elements are the connected components of $H_2$, which is shown in the left of \Cref{fig:ex:sss}. 
Thus, $\P(\A)$ is strictly supersolvable.

Note, however, that the fiber bundle depicted in \Cref{fig:loctriv} corresponds to the M-ideal $\P(\A)_{\leq H_{\alpha_1}}$ in the right of \Cref{fig:ex:sss}, which is not a TM-ideal since $H_{\alpha_1}\cap H_{\alpha_3}$ is disconnected.
\begin{figure}[ht]
\begin{tikzpicture}
\solidnodes
\node[red] (T) at (0,0) {};
\node (1) at (-1.5,1.5) {};
\node[red] (21) at (-.5,1.5) {};
\node[red] (22) at (.5,1.5) {};
\node (3) at (1.5,1.5) {};
\node (11) at (-1,3) {};
\node (-11) at (1,3) {};
\draw[-] (T)--(1);
\draw[red,-,ultra thick] (T)--(21);
\draw[red,-,ultra thick] (T)--(22);
\draw[-] (T)--(3);
\draw[-] (1)--(11);
\draw[-] (1)--(-11);
\draw[-] (21)--(11);
\draw[-] (22)--(-11);
\draw[-] (3)--(11);
\draw[-] (3)--(-11);
\end{tikzpicture}
\hspace{1in}
\begin{tikzpicture}
\solidnodes
\node[blue] (T) at (0,0) {};
\node (1) at (-1.5,1.5) {};
\node (21) at (-.5,1.5) {};
\node (22) at (.5,1.5) {};
\node[blue] (3) at (1.5,1.5) {};
\node (11) at (-1,3) {};
\node (-11) at (1,3) {};
\draw[-] (T)--(1);
\draw[-] (T)--(21);
\draw[-] (T)--(22);
\draw[blue,-,ultra thick] (T)--(3);
\draw[-] (1)--(11);
\draw[-] (1)--(-11);
\draw[-] (21)--(11);
\draw[-] (22)--(-11);
\draw[-] (3)--(11);
\draw[-] (3)--(-11);
\end{tikzpicture}
\caption{The Hasse diagram for the poset of layers of the arrangement $\A$ from Example \ref{ex:sss}. The elements of a TM-ideal colored in \textcolor{red}{red} (left); the elements of an M-ideal that is not a TM-ideal are colored in \textcolor{blue}{blue} (right).}
\label{fig:ex:sss}
\end{figure}
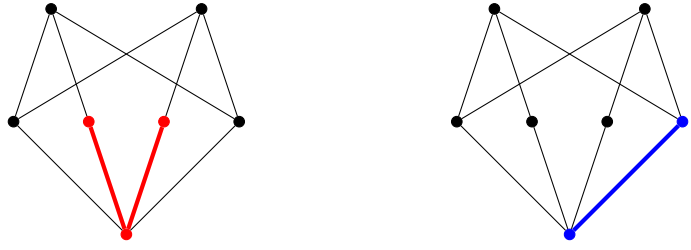
\end{example}

\begin{example}
Dowling posets are strictly supersolvable; see the proof of \Cref{prop:dowling}.
\end{example}

\begin{proposition}
Let $\P$ be a geometric semilattice. Then $\Q$ is an M-ideal of $\P$ if and only
if $\Q$ is a TM-ideal of $\P$. Consequently, $\P$ is supersolvable if and only if it
is strictly supersolvable.
\end{proposition}
\begin{proof}
In a geometric semilattice, if two elements have an upper bound then they have a
unique minimal upper bound. This means that \Cref{def:mod:1} implies
\Cref{def:TM:1} for geometric semilattices, hence every M-ideal is a TM-ideal.
\end{proof}

\subsection{Characteristic polynomial}
The {\em characteristic polynomial} of any bounded-below poset $\P$ with a rank
function $\rk$ is defined as
\[ \chi_{\P}(t):=\sum_{x\in \P} \mu_{\P}(x)t^{\rk(\P)-\rk(x)}, \]
where $\mu_{\P}$ is iteratively defined by $\mu_{\P}(\zero)=1$ and for any
$x\in\P\setminus\zero$ one has $\sum_{y\in\P_{\leq x}} \mu_{\P}(y) = 0$.
A feature of supersolvable geometric lattices is that their characteristic
polynomial decomposes into linear factors over $\mathbb{Z}$. We show that this
is true also for \textit{strictly} supersolvable posets.

\begin{theorem}\label{thm:charpoly}
Let $\Q $ be a TM-ideal of a locally geometric poset $\P$ with
$\rk(\Q )=\rk(\P)-1$, and let $a=|A(\P)\setminus A(\Q )|$. Then
\[\chi_\P(t) = \chi_\Q (t)\cdot(t-a).\]
In particular, if $\P$ is strictly supersolvable via the chain of TM-ideals $\zero=\Q_0\subset \Q_1\subset\cdots\subset Q_n=\P$, and $a_i=|A(\Q_i)\setminus A(\Q_{i-1})|$ for each $i$, then 
\[\chi_\P(t) = \prod_{i=1}^n (t-a_i).\]
\end{theorem}
\begin{proof}
The second claim, for strictly supersolvable posets, follows from the first by induction. Thus, we need only consider a TM-ideal $\Q$ in a locally geometric poset $\P$.

Let $x\in\P\setminus\Q$. By \Cref{prop:localmod} and \Cref{rem:uniquey}, there is a unique $y\in\Q$ such that $x$ covers $y$, and this $y$ is a modular element of $\P_{\leq x}$. Thus, $\chi_{\P_{\leq x}}(t)=\chi_{\Q_{\leq y}}(t)(t-|A_x|)$ by \cite[Theorem 2]{stanley}, where $A_x=A(\P_{\leq x})\setminus A(\Q)$. Extracting the constant term of each side of this equation yields $\mu_\P(x)=-\mu_\Q(y)|A_x|$.
Moreover, for $y\in\Q$, the sets $\{A_x \st x\in\P\setminus\Q, y\lessdot x\}$ partition $A(\P)\setminus A(\Q)$ by \Cref{def:TM:1}. This means that for each $y\in\Q$,  $\displaystyle\sum_{\substack{x\in\P\setminus\Q\\ y\lessdot x}} |A_x|=a$.

Therefore,
\begin{align*}
\chi_{\P}(t)
&=\sum_{x\in\P} \mu_{\P}(x) t^{\rk(\P)-\rk(x)}\\
&= \sum_{y\in \Q } \mu_\Q (y)t^{\rk(\Q )+1-\rk(y)} + 
\sum_{y\in \Q } \sum_{\substack{x\in\P\setminus \Q \\y\lessdot x}} \mu_{\P}(x)
t^{\rk(\P)-\rk(x)}\\
&= \sum_{y\in \Q } \mu_\Q (y)t^{\rk(\Q )+1-\rk(y)} + 
\sum_{y\in \Q } \sum_{\substack{x\in\P\setminus \Q \\y\lessdot x}} -\mu_{\Q}(y)|A_x|
t^{\rk(\P)-\rk(x)}\\
&= \sum_{y\in \Q } \mu_\Q (y)t^{\rk(\Q )+1-\rk(y)} + 
\sum_{y\in \Q } -a\mu_\Q (y) t^{\rk(\Q )-\rk(y)}\\
&= \sum_{y\in \Q } \mu_\Q (y) t^{\rk(\Q )-\rk(y)} (t-a) 
= \chi_\Q (t)\cdot(t-a).
\end{align*}
\end{proof}

\begin{remark}
The complete factorization in the case of strictly supersolvable posets
could be proved directly using \cite[Theorem 4.4]{Hallam} and the partition of $A(\P)$ whose blocks are the sets $A(\Q_i)\setminus A(\Q_{i-1})$.
This was carried out in the example of Dowling posets in \cite[Theorem B]{BG} and is straightforward to generalize to our setting.
\end{remark}

\begin{remark}
The assumption that $\Q $ is a TM-ideal in \Cref{thm:charpoly} is
necessary, as demonstrated in the following examples. Accordingly, a poset being
supersolvable is not enough for its characteristic polynomial to factor
completely over $\Z$.
\end{remark}

\begin{example}\label{ex:charpoly}
Consider the poset $\P$ depicted in \Cref{fig:ex:sss} (see also \Cref{ex:sss}).
Its characteristic polynomial is
\[\chi_{\P}(t) = t^2-4t+4 = (t-2)(t-2).\]
This agrees with the fact that the TM-ideal $\Q$ colored in red in \Cref{fig:ex:sss} has $\chi_\Q(t)=t-2$ and $|A(\P)\setminus A(\Q)|=2$.
On the other hand, the M-ideal $\Q'$ colored in blue in \Cref{fig:ex:sss} is not a TM-ideal,
and we see that $\chi_{\Q '}(t)=t-1$ does not divide $\chi_{\P}(t)$.
\end{example}

\begin{example}
Consider again the poset $\P$ in \Cref{fig:ex:ss:mod}, see also 
\Cref{ex:ss:mod}. It is supersolvable, with $\{\zero,1\}$ and $\{\zero,3\}$
both M-ideals. However, it is not strictly supersolvable and its
characteristic polynomial $\chi_{\P}(t) = t^2-3t+3$ does not factor over the
integers. 
\end{example}

\begin{corollary}\label{cor:poincare}
Let $\G\cong(S^1)^d\times\R^v$ with $v>0$, let $\Gamma$ be a free abelian group with
$\rk\Gamma=n$, let $\A$ be an essential arrangement in $T=\Hom(\Gamma,\G)$ whose
poset of layers is strictly supersolvable via the chain of TM-ideals
$\zero\subset \Q_1\subset\cdots\subset \Q_n=\P(\A)$, and let $a_i=|A(\Q_i)\setminus
A(\Q_{i-1})|$ for each $i$. 

Then the Poincar\'e polynomial of $M(\A)$ factors as:
\[\sum_{j\geq0} \rk H_j(M(\A);\Z) t^j = \prod_{i=1}^n \left( (1+t)^d + a_i
t^{d+v-1}\right).\]
\end{corollary}
\begin{proof}
By \cite[Theorem 7.8]{LTY}, the Poincar\'e polynomial of $M(\A)$ is given by
\[ \sum_{j\geq0} \rk H_j(M(\A);\Z) t^j = (-t^{d+v-1})^n \chi_{\P(\A)}\left(
-\frac{(1+t)^d}{t^{d+v-1}}\right).\]
The result then follows from the formula of \Cref{thm:charpoly}.
\end{proof}

\subsection{Topological fibrations}
Recall from \Cref{ex:graphic} the configuration space $\Conf_n(\G)$, viewed as the complement of an abelian arrangement.
Cohen \cite[Theorem 1.1.5]{cohen} observed that fiber bundles of hyperplane arrangement complements can be pulled back from the classical Fadell--Neuwirth bundles of configuration spaces \cite{FN} defined by dropping the last point in a configuration.
We see the same phenomenon when our arrangement bundle corresponds to a TM-ideal, but an M-ideal is not sufficient. 

As in \S\ref{sec:fib}, we assume throughout that $\G$ is a connected abelian Lie group, $T=\Hom(\Gamma,\G)$ for some finite-rank free abelian group $\Gamma$, $Y$ is an admissible subgroup of $T$, and $\A$ is an essential arrangement in $T$.

\begin{theorem}\label{thm:FN}
Suppose that $\P(\A_Y)$ is a TM-ideal of $\P(\A)$. Then 
there exists a map $g:M(\qA)\to\Conf_\ell(\G)$ such that $p$ is the
pullback of the Fadell--Neuwirth bundle $\Conf_{\ell+1}(\G)\to\Conf_\ell(\G)$
along the map $g$.
\end{theorem}
\begin{proof}
Write $\A\setminus\A_Y=\{\alpha_1,\dots,\alpha_s\}$, and for each $1\leq i\leq s$ write $H_{\alpha_i}=\sqcup_{j=1}^{c_i} H_{\alpha_i,j}$, where $c_i$ is the number of connected components of $H_{\alpha_i}$. 
By \Cref{def:TM:1}, the covering map in \Cref{cor:cover} restricts to a homeomorphism $p_{i,j}:=p|_{H_{\alpha_i,j}}:H_{\alpha_i,j}\to T/Y$ for each $1\leq i\leq s$ and $1\leq j\leq c_i$. 
 We can choose coordinates (e.g., as in the proof of \Cref{lem:calpha}) such that $p_{i,j}$ is the projection along the last coordinate, sending $(t,r)\mapsto t$. Since $p_{i,j}$ is a homeomorphism, there are continuous functions $r_{i,j} : T/Y\to \G$ such that $p_{i,j}^{-1}:T/Y\to H_{\alpha_i,j}$ sends $t\mapsto (t,r_{i,j}(t)) $. 
 Then the product of the $r_{i,j}$ defines a continuous map $\hat{g}:T/Y\to\G^\ell$,  
 where $\ell=\sum_i c_i$.
\Cref{lem:calpha} guarantees that for any $t\in M(\qA)$, the points $r_{i,j}(t)$ are all distinct. Therefore, $\hat{g}$ restricts to a continuous map $g:M(\qA)\to \Conf_\ell(\G)$.

Notice that the two fiber bundles form a commutative square as depicted in \Cref{fig:FN}, where the map $h:M(\A)\to \Conf_{\ell+1}(\G)$ is defined, for $(t,z)\in M(\A)\subseteq T\cong  (T/Y)\times \G$, by $(t,z)\mapsto (g(t),z)$. We now show that this square satisfies the universal property of a pullback.
 
\begin{figure}[ht]
\begin{tikzcd}
X \arrow[bend left=10]{rrd}{f_2}\arrow[dashed]{rd}{f}\arrow[bend right=10, swap]{ddr}{f_1} &&\\
& M(\A) \arrow{r}{h}\arrow{d}{p}& \Conf_{\ell+1}(\G)\arrow{d}{\pi} \\
& M(\qA) \arrow{r}{g} & \Conf_\ell(\G) 
\end{tikzcd}
\caption{Pullback diagram in the proof of \Cref{thm:FN}}
\label{fig:FN}
\end{figure}
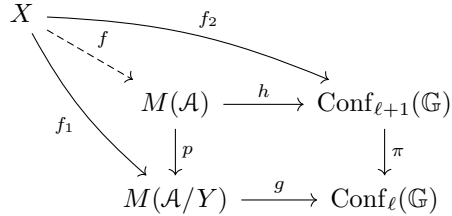
Let $X$ be a topological space and $f_1,f_2$ two continuous functions such that $g\circ f_1=\pi\circ f_2$. The map
$$
f: X\to M(\A),\quad\quad x\mapsto (f_1(x),(f_2(x))_{\ell+1}) 
$$
is well-defined because $f_1(x)\in M(\qA)$ and $(f_2(x))_{\ell+1}$ does not coincide with any of  the coordinates of $g(f_1(x))=(\pi\circ f_2)(x)=(f_2(x)_1,\ldots f_2(x)_{\ell})$, i.e.\ the punctures of the fiber $p^{-1}(t)$. It is continuous because the $f_i$ are, and commutes with $p$ and $h$, since by definition $p\circ f = f_1$ and $h\circ f = (g\circ f_1, (f_2)_{\ell+1}) = (\pi\circ f_2, (f_2)_{\ell+1}) = f_2$.
\end{proof}

\begin{remark}\label{rem:monodromy}
The idea to have in mind for the map $g$ of \Cref{thm:FN} is that it should pick
out the punctures in the fiber over a point of $M(\qA)$.
If $\P(\A_Y)$ is an M-ideal but not a TM-ideal, then there is some $H\in
A(\A\setminus\A_Y)$ such that $H\cap Y$ is not connected. The nontrivial monodromy of the covering map $H\to T/Y$ then obstructs the existence of the desired map $r$ within the proof.
We will see below (\Cref{thm:monodromy}) that trivial monodromy implies that the M-ideal $\P(\A_Y)$ is in fact a TM-ideal.
\end{remark}

\begin{corollary}\label{cor:section}
If $\P(\A_Y)$ is a TM-ideal of $\P(\A)$, then the fiber bundle $M(\A)\to M(\qA)$ admits a section.
\end{corollary}
\begin{proof}
By \Cref{thm:FN} it suffices to prove that the  bundle $\Conf_{\ell+1}(\G)\to\Conf_\ell(\G)$ has a section. For $v>0$ this follows from \cite[Example 4.1]{cohen-config}.
For the general case, note that this bundle is isomorphic to the bundle $\pi\times\operatorname{id}:\Conf_{\ell}(\G\setminus0)\times\G \to \Conf_{\ell-1}(\G\setminus0)\times\G$ (see eg. \cite[Theorem 4]{FN}). In order to obtain a section of $\pi$, fix a coordinate direction in $\G$. Given $x=(x_1,\ldots,x_{\ell-1})\in\Conf_{\ell-1}(\G\setminus 0)$ let $\mu(x)$ be the minimum of the (finite) set of (nonzero) distances of the $x_i$ from $0$ in $\G$. Then a section of $\pi$ is obtained by mapping $x$ to $(x_1,\ldots,x_{\ell-1},x_\ell)$, where $x_\ell$ is the point of $\G$ at distance $\frac{\mu(x)}{2}$ from $0$ in the chosen coordinate direction.
\end{proof}

\begin{corollary}\label{cor:pi1}
The fundamental group of a strictly supersolvable linear, toric, or elliptic arrangement has the structure of an iterated semidirect product of free groups.
\end{corollary}
\begin{proof}
We proceed by induction on the rank of $\Gamma$. If $\rk(\Gamma)=1$, then $M(\A)\cong\G\setminus\{\text{finitely many points}\}$ with $\G=\C$, $\C^\times$, or $(S^1)^2$, hence the fundamental group of $M(\A)$ is isomorphic to a free group.
Now for $\rk(\Gamma)>1$, there is a strictly supersolvable arrangement $\B$ for which $M(\A)\to M(\B)$ is a fiber bundle with $F\cong\G\setminus\{\text{finitely many points}\}$. Using the homotopy long exact sequence for the fibration and \Cref{cor:kpi1}, we have a short exact sequence of fundamental groups
\[ 0 \to \pi_1(F) \to \pi_1(M(\A)) \to \pi_1(M(\B)) \to 0.\]
\Cref{cor:section} implies that this short exact sequence is split, and hence \[\pi_1(M(\A))\cong\pi_1(F)\rtimes\pi_1(M(\B)).\] Since $\pi_1(F)$ is a free group the result follows by induction.
\end{proof}

\begin{lemma}\label{map_pi1}
Assume that $\G\cong (S^1)^d\times\R^v$ and $v>0$.
The inclusion $\iota: M(\A)\to T$ induces a surjective map of fundamental groups $\iota^*:\pi_1(M(\A))\to\pi_1 (T)$.
\end{lemma}
\begin{proof}
The product structure of $\G$ induces a decomposition \[\Hom(\Gamma,\G)=\Hom(\Gamma,(S^1)^d)\times\Hom(\Gamma,\R^{v}).\]
Thus, every $t\in \Hom(\Gamma,\G)$ is a pair $t=(z, x)$ with $z\in \Hom(\Gamma,(S^1)^d)$ and $x\in \Hom(\Gamma,\R^v)$. 
Now $\alpha\in \Gamma$ is in $\ker(t)$ if and only if $\alpha\in \ker(z)\cap\ker(x)$.
Let $M(\A_\R)$ denote the complement of the arrangement defined by $\A$ in $\Hom(\Gamma,\R^v)\simeq \R^{vn}$. Let $r\in M(\A_\R)$. Then for all $\alpha\in \A$ we have $\alpha\not\in\ker(r:\Gamma\to\R^v)$  and so $\alpha\not\in\ker((z,r):\Gamma\to(S^1)^d\times\R^v)$ for every $z\in \Hom(\Gamma,(S^1)^d)\simeq (S^1)^{dn}$. 
  Thus $\pi_\R^{-1}(r)\subseteq M(\A)$, where $\pi_\R$ denotes the natural projection $T\to \R^{vn}$ induced by the projection $\G\to\R^v$.
By definition of $\pi_\R$, we have $T\simeq \pi_\R^{-1}(r)\times\R^{vn}$, hence the inclusion $j:\pi_\R^{-1}(r)\hookrightarrow T$ induces an isomorphism of fundamental groups. Since $j$ factors as a composition of inclusions $T_c\hookrightarrow M(\A)\hookrightarrow T$, the homomorphism $\iota^*:\pi_1(M(\A))\to \pi_1(T)$ is surjective as claimed. 
\end{proof}

\begin{theorem}\label{thm:monodromy}
Assume that $\G\cong(S^1)^d\times\R^v$ and $v>0$. Let $Y$ be an admissible subgroup of $T$ for which $M(\A)\to M(\qA)$ is a fiber bundle with fiber $F$. If the monodromy action of $\pi_1(M(\qA))$ on $H^*(F;\Z)$ is trivial, then $\P(A_Y)$ is a TM-ideal of $\P(\A)$. Consequently, the conclusions of \Cref{thm:FN} and \Cref{cor:section} hold, as well as the following tensor decomposition of vector spaces:
\[H^*(M(\A);\QQ) \cong H^*(M(\qA);\QQ)\otimes H^*(F;\QQ).\]
\end{theorem}
\begin{proof}
The conclusions of \Cref{thm:FN} and \Cref{cor:section} will hold as long as $\P(\A_Y)$ is a TM-ideal of $\P(\A)$.
The triviality of the monodromy action implies that the $E_2$ term of the Serre spectral sequence of the fiber bundle $M(\A)\to M(\qA)$ can be expressed as the above tensor product. Since this tensor product has Hilbert series equal to that of $H^*(M(\A);\QQ)$ via \Cref{cor:poincare}, the differentials must be trivial.
It remains to prove that if the monodromy action is trivial, then $\P(\A_Y)$ is a TM-ideal of $\P(\A)$.

By \Cref{lem:calpha} and \Cref{rem:TM}, if 
$\P(\A_Y)$ is not a TM-ideal, 
then there is $H\in \A\setminus \A_Y$ such that $H\cap Y$ is disconnected, hence it contains two 
points $y_0,y_1$ in distinct connected components. Since $H$ is path-connected, we can let $\gamma$ denote a path from $y_0$ to $y_1$ in $H$. Then, $p(\gamma)$ is a closed path in $T/Y$. The element of $\pi_1(T/Y)$ defined by the class of $p(\gamma)$ determines a continuous map $f:Y\to Y$ with $f(y_0)=y_1$ (by uniqueness of lifting). 

The inclusions $\iota: M(\qA) \hookrightarrow T/Y$ and $M(\A) \hookrightarrow T$ define a bundle map from the fibration $M(\A)\to M(\qA)$ with fiber $F$  
to the trivial fibration $T\to T/Y$ with fiber $Y$.

By \Cref{map_pi1}, there is a closed path $\gamma'$ in $M(\qA)$ whose homotopy class maps to the class of $p(\gamma)$ under $\iota^*$. The element of $\pi_1(M(\qA))$ defined by $\gamma'$ acts on $F$ via $f':=f_{\vert F}$, the restriction of $f$. Now it is enough to prove that $f'$ acts nontrivially on the cohomology of $F$.

Write $k:=\dim \G$.  The pair $(Y,F)$ has the (co)homology of a wedge of $k$-spheres, one for each puncture of $F$. Thus we can consider the generators $e_0,e_1$ of $H^k(Y,F)$ corresponding to $y_0,y_1$. The map $g$ induced by $f$ on  $H^{k}(Y,F)$ satisfies $g(e_0)={e_1}$.  
The functions $f, f', g$ define the following automorphism of the long exact sequence of the pair $(Y,F)$. (Note that $v>0$ implies that the homological dimension of $\G$ -- hence of $Y$ -- is strictly less than $k$.)

\begin{center}
\begin{tikzcd}
\cdots \arrow[r] 
& H^{k-1}(F) \arrow[r, "\partial"]\arrow[d, "(f')^*"] & 
H^{k}(Y,F) \arrow[r] \arrow[d, "g"] & 
H^k (Y) = 0 \arrow[r]\arrow[d]  &\cdots \\
\cdots \arrow[r] 
& H^{k-1}(F) \arrow[r] & H^{k}(Y,F) \arrow[r] & H^k (Y) = 0 \arrow[r] &\cdots \\
\end{tikzcd} 
\end{center}
Since $e_0\in H^{k}(Y,F)$ is in the kernel of the differential, there is $a\in H^{k-1}(F)$ with $\partial(a)=e_0$. Now  $(f')^*(a)$ must map under $\partial$ to $e_1=g(e_0)$. Since $e_0\neq e_1$, we have $a\neq (f')^*(a)$, showing that the action of $\gamma'\in\pi_1(\qA)$ on $H^{k-1}(F)$ is nontrivial.
\end{proof}

\begin{remark}\label{rmk:toricmonodromy}
Let $\G=S^1\times\R$. Then $\Conf_\ell(\G)$ is equal to the complement of a fiber-type hyperplane arrangement, namely the arrangement in $\C^\ell$ containing all diagonal and coordinate hyperplanes. As such, the monodromy action of $\pi_1(\Conf_{\ell-1}(\G))$ on the cohomology of the fiber $\G\setminus\{\ell-1\text{ points}\}$ is trivial \cite[Proposition 2.5]{FR}.
Using \Cref{thm:monodromy}, this implies a converse of \Cref{thm:FN}: {\em whenever a toric arrangement bundle $M(\A)\to M(\qA)$ is pulled back from a configuration space bundle $\Conf_{\ell}(\G)\to \Conf_{\ell-1}(\G)$, $\P(\A_Y)$ is a TM-ideal of $\P(\A)$.}
\end{remark}

\begin{question}
\label{q:monodromy}
Let $\G$ be any connected abelian Lie group and $\Gamma$ a finitely generated free abelian group. Let $\A$ be an arrangement in $T=\Hom(\Gamma,\G)$, and let  $Y$ be an admissible subgroup of $T$ for which $M(\A)\to M(\qA)$ is a fiber bundle with fiber $F$. What is the monodromy action of $\pi_1(M(\qA))$ on the cohomology of the fiber $F$?
\end{question}

We conclude with a formula relating the cohomology of a fiber-type toric arrangement with the lower central series of its fundamental group. For this, recall the lower central series of a group.

\begin{definition}\label{def:LCS}
Let $\gp$ be a group. The \defh{lower central series} of $\gp$ is defined as $\gp=\gp_1\supseteq\gp_2\subseteq\gp_3\supseteq\cdots$ where $\gp_{j+1}=[\gp_j,\gp]$ for all $j\geq1$.
For each $j$, let $\gp(j):=\gp_j/\gp_{j-1}$.
When $\gp$ is a finitely generated group, each $\gp(j)$ is a finitely-generated abelian group \cite[Theorem 5.4]{MKS}. Then, we let $\varphi_j(\gp)$ denote the rank of $\gp(j)$ as an abelian group.
\end{definition}
Note that $\gp(1)$ is the abelianization of $\gp$. In particular, if $X$ is a topological space and $\gp=\pi_1(X)$ is finitely generated, then $\gp(1)\cong H_1(X)$ and $\varphi_1$ is the first Betti number of $X$.

\begin{theorem}\label{thm:LCS}
Let $\A$ be a strictly supersolvable toric arrangement via the chain of TM-ideals $\P(\A)=\Q_n\supset\Q_{n-1}\supset\cdots\supset\Q_1\supset\Q_0=\zero$. 
Let $a_i=|A(\Q_i)\setminus A(\Q_{i-1})|$ for each $i=1,2,\dots,n$.
Let $\gp=\pi_1(M(\A))$ be the fundamental group of the arrangement complement, and for each $j\geq 1$ abbreviate $\varphi_j=\varphi_j(\gp)$.
Then each $\gp(j)$ is a free abelian group and 
\[ \prod_{j=1}^\infty (1-t^j)^{\varphi_j}
= \prod_{i=1}^n (1-(a_i+1)t),\]
which by \Cref{cor:poincare} is equal to the Poincar\'e polynomial of $M(\A)$.
\end{theorem}
\begin{proof}
We proceed by induction on $n$, with both the base case and the inductive step relying on the following identity from \cite[p. 330 and Corollary 5.12.(iv)]{MKS}. Let $\mathfrak{F}_r$ be the free group on $r$ generators, and let $N_{r,j} = \varphi_j(\mathfrak{F}_r)$. Then 
\begin{equation}\label{eq:MKS}\tag{$\star$}
\prod_{j=1}^\infty (1-t^j)^{N_{r,j}} = 1-rt.
\end{equation}

Now, our base case is an arrangement of $a_1$ points in $\C^\times$, whose complement is the complement of $a_1+1$ points in $\C$.
Thus, the fundamental group is a free group on $r=a_1+1$ generators, for which the groups $\mathfrak{F}_r(j)$ are free abelian and \eqref{eq:MKS} is precisely the desired formula.

Now for $n>1$, let $Y\in\P(\A)$ be the subgroup for which $\Q_{n-1}=\P(\A_Y)$. 
Let $F$ be the fiber of the bundle $M(\A)\to M(\qA)$, so that $\pi_1(F)\cong\mathfrak{F}_r$ with $r=a_n+1$.
By \Cref{cor:kpi1} and \Cref{cor:section}, we have a split short exact sequence of fundamental groups 
\begin{center}
\begin{tikzcd}
1 \arrow[r]
& \pi_1(F) \arrow[r]
& \pi_1(M(\A)) \arrow[shift left,r]
& \pi_1(M(\qA)) \arrow[shift left,l] \arrow[r]
& 1 
\end{tikzcd} 
\end{center}
Abbreviate $\gp=\pi_1(M(\A))$, $\varphi_j=\varphi_j(\gp)$, $\gp'=\pi_1(M(\qA))$, and $\varphi'_j=\varphi_j(\gp')$.
Then \cite[Corollary 3.6]{FR} implies
$\gp(j)\cong\gp'(j)\oplus\mathfrak{F}_r(j)$ for every $j\geq 1$.
By our inductive hypothesis, each $\gp(j)$ is a direct sum of free abelian groups, hence it is also free abelian. 
This direct sum decomposition also implies
$\varphi_j=\varphi'_j+N_{r,j}$ for every $j\geq 1$, thus
\[
\prod_{j=1}^\infty (1-t^j)^{\varphi_j}
= \prod_{j=1}^\infty(1-t^j)^{\varphi'_j}
\prod_{j=1}^\infty(1-t^j)^{N_{r,j}}
\]
which by induction and \eqref{eq:MKS} is equal to
\[\prod_{i=1}^{n-1}(1-(a_i+1)t)(1-(a_n+1)t).\]
\end{proof}


\end{document}